\documentclass[a4paper]{amsart}
\usepackage{amsmath,amsthm,amssymb,latexsym,epic,bbm,comment,mathrsfs}
\usepackage{graphics,enumerate,stmaryrd,color,tikz}
\usepackage[all,2cell]{xy}
\xyoption{2cell}
\usepackage[most]{tcolorbox}

\newtheorem{theorem}{Theorem}
\newtheorem{lemma}[theorem]{Lemma}
\newtheorem{corollary}[theorem]{Corollary}
\newtheorem{proposition}[theorem]{Proposition}

\newtheorem{remark}[theorem]{Remark}
\usepackage[all]{xy}
\usepackage[active]{srcltx}
\usepackage[parfill]{parskip}
\usepackage{enumerate}

\usepackage{pgfplots}
\usepgfplotslibrary{groupplots}
\pgfplotsset{compat=1.18}

\begin{document}
\title[On the number of extension closed additive subcategories]{On the 
number of extension closed additive subcategories for
uniformly oriented $A_n$ quivers}

\author[V.~Mazorchuk]
{Volodymyr Mazorchuk}

\begin{abstract}
We provide a recurrence for computing the 
terms of the OEIS sequence A393920, 
introduced in \cite{KS}. 
We also describe a surprising connection between 
A393920 and the Fibonacci sequence A000045, 
obtain non-trivial lower and upper exponential 
bounds for its growth, and investigate connections 
with partial orders, Catalan numbers, and convex 
topologies on finite chains.  For the representation-theoretic
lattice underlying A393920, we describe its atoms,
coatoms, join-irreducible and meet-irreducible elements.
\end{abstract}

\maketitle

\section{Description of the results and history of the paper}\label{s1}

\subsection{Combinatorial model}\label{s1.1}

Let $n$ be a non-negative integer. Define $\mathcal{Q}(n)$ as the set of
all pairs $(x,y)\in\mathbb{Z}^2$ satisfying the following conditions:
\begin{displaymath}
0\leq x,\qquad 0\leq y,\qquad x+y\leq n.
\end{displaymath}
Here is a visual depiction of $\mathcal{Q}(n)$
(the points are in {\color{blue}blue}), for $n=0,1,2,3$:

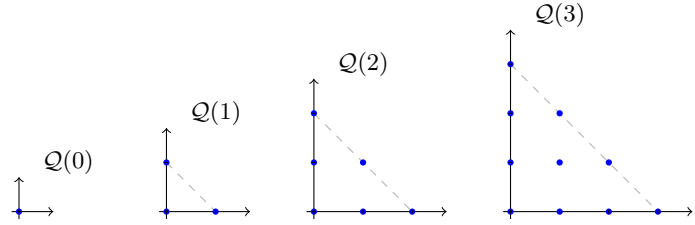
\begin{center}
\begin{tikzpicture}[scale=0.65]
\newcommand{\drawQ}[3]{%
  \begin{scope}[shift={(#2,#3)}]
    \def\n{#1}   
    \node[font=\small] at (1,\n+1) {$\mathcal{Q}(#1)$};
    \draw[->] (-0.15,0) -- (\n+0.7,0);
    \draw[->] (0,-0.15) -- (0,\n+0.7);
    \draw[gray!60,dashed] (0,\n) -- (\n,0);
    \foreach \x in {0,...,#1} {
      \foreach \y in {0,...,#1} {
        \pgfmathtruncatemacro{\s}{\x+\y}
        \ifnum\s>#1
        \else
          \fill[blue] (\x,\y) circle (1.8pt);
        \fi
      }
    }
    \foreach \k in {0,...,#1} {
      \draw (\k,0.04) -- (\k,-0.04);
      \draw (0.04,\k) -- (-0.04,\k);
    }
  \end{scope}
}

\drawQ{0}{0}{0}
\drawQ{1}{3}{0}
\drawQ{2}{6}{0}
\drawQ{3}{10}{0}
\end{tikzpicture}
\end{center}

Define the set $\mathcal{R}(n)$ as the set of all subsets $Y$
of $\mathcal{Q}(n)$ satisfying the following condition,
which we refer to as Condition~$(\star)$ in the paper:
\begin{itemize}
\item[($\star$)] For any $(a,b),(c,d)\in Y$ such that 
$\max(a,c)+\max(b,d)\leq n+1$, we have
$(\min(a,c),\min(b,d))\in Y$
and, additionally, $(\max(a,c),\max(b,d))\in Y$
as soon as $(\max(a,c),\max(b,d))\in \mathcal{Q}(n)$.
\end{itemize}
Set $R_n$ to be the cardinality of $\mathcal{R}(n)$.

For example, $\mathcal{R}(0)=\boldsymbol{2}^{\mathcal{Q}(0)}=\{\varnothing,\{(0,0)\}\}$ while
the only subset of $\mathcal{Q}(1)$ which does not belong to
$\mathcal{R}(1)$ is $\{(0,1),(1,0)\}$. The latter is explained by
the following observation:
$\max(0,1)+\max(1,0)\leq 1+1$ but 
$(\min(0,1),\min(1,0))\not\in \{(0,1),(1,0)\}$.
Therefore $R_0=2$ and $R_1=7$.

Define
\begin{displaymath}
\mathcal{P}(n):=\{X\in \mathcal{R}(n)\,:\, (0,0)\in X\} 
\end{displaymath}
and set $P_n:=|\mathcal{P}(n)|$. For example, $P_0=1$ while $P_1=4$.

\subsection{Recurrence}\label{s1.2}

Define two functions $a,b:\mathbb{Z}^2\to \mathbb{Z}$ recursively.
As the initial conditions for $a$, we set 
\begin{itemize}
\item $a(n,i)=0$ if $i<0$ or $n<-2$;
\item $a(n,i)=0$ if $n\geq 0$ and $i>n+1$;
\item $a(n,i)=0$ if $n<0$ and $i\neq 0$;
\item $a(-1,0)=a(-2,0)=a(0,0)=a(0,1)=1$.
\end{itemize}
As the initial conditions for $b$, we set 
\begin{itemize}
\item $b(n,i)=0$ if $i\leq 0$ or $i>n+1$ or $n<0$;
\item $b(0,1)=1$.
\end{itemize}
Consequently, the two functions are supported on a triangular cone.

The recurrences interlace the two functions:
\begin{equation}\label{recurrence:a1}
a(n,0)=\sum_{i=0}^{n}a(n-1,i),
\end{equation}
for $n>0$, and
\begin{equation}\label{recurrence:a2}
a(n,k)=b(n,k)+\sum_{j=0}^{n-1} b(j,k)\sum_{i=0}^{n-j-1}a(n-j-2,i),
\end{equation}
for $n>0$ and $k>0$. Furthermore,
\begin{equation}\label{recurrence:b}
b(n,k)=a(n-1,k-1)+\sum_{m=1}^{n}\sum_{r=0}^{\min(m-1,k-1)} a(m-2,r)\sum_{q=k-1-r}^{n-m+1} b(n-m,q),
\end{equation}
for $n>0$ and $1\leq k\leq n+1$.

Define two functions $A,B:\mathbb{Z}\to \mathbb{Z}$ as
\begin{equation}\label{abvsAB}
A(n):=\sum_{i\in\mathbb{Z}}a(n,i)\qquad\text{ and }\qquad
B(n):=\sum_{i\in\mathbb{Z}}b(n,i).
\end{equation}

\subsection{Main result}\label{s1.3}

The main result of the paper is the following:

{\bf Theorem A}. For $n\geq 0$, we have $R_n=A(n)$ and $P_n=B(n)$.

Due to the easy recursive nature of both $A(n)$ and $B(n)$, this 
allows for very efficient computations of the values of 
$R_n$ (the OEIS sequence A393920) and $P_n$ (the OEIS sequence A234268).

\subsection{Other results}\label{s1.4}

We also explore various connections between $R_n$ and $P_n$
and several integer sequences:

\begin{itemize}
\item In Section~\ref{s3} we define a certain quotient of 
the restriction graph for $\mathcal{R}(n)$ and in 
Section~\ref{s4} we show that vertices in this quotient
graph are enumerated by the classical Fibonacci numbers.
\item In Subsection~\ref{s9.3} we show that 
$a(n,n+1)$ is the classical Catalan sequence.
\item In Subsection~\ref{s9.5} we show that $P_n$
enumerates convex topologies on a finite totally ordered set,
which is the sequence A234268  in \cite{OEIS}.
\end{itemize}

In Section~\ref{s7}, we provide some estimates for the
growth of $R_n$. This includes an asymptotic 
exponential lower bound with base $9$ 
and an asymptotic exponential upper bound with base
$\approx 12.75$.

The set $\mathcal{R}(n)$ is a lattice with respect to inclusions.
In Section~\ref{s12}, we study some properties related to this structure.
In particular, we describe atoms, coatoms, join-irreducible and
meet-irreducible elements. Some of these are very easy, 
some other are not.

\subsection{How this paper originated and was written
with help of Microsoft Copilot}\label{s1.5}

The sequence A393920 was added to \cite{OEIS} very recently.
It computes the number of extension closed, additive, idempotent 
split subcategories of the category of finite dimensional 
representations of a uniformly oriented $A_{n+1}$ quiver, see \cite{KS}. 
This is exactly the sequence $R_n$ defined above.

Henning Krause mentioned  A393920 during his talk at Uppsala 
University in May 2026
which got me interested. Despite of an  elementary and easy 
description, it seemed that computation of the elements of this 
sequence was a challenging problem. In \cite{KS}, the values were 
computed up to $n=6$, with two more values
added by Fr{\'e}d{\'e}ric Chapoton, as can be seen on 
the corresponding page of \cite{OEIS}.

Since the task 
sounded elementary but computationally challenging, 
I first tried to discuss the sequence with
Microsoft Copilot. After a few back and forth, Copilot proposed an 
effectivization for the naive algorithm for computation of the sequence,
computed a few new values and made an interesting observation of 
potential connection to the Fibonacci sequence A000045.
The present note originated as an attempt of a rigorous write-up of these 
discussions with Copilot.

The point of Copilot's effectivization was the following: the naive approach
to compute the sequence would be to go through all subsets of our 
triangular array of integer points in the plane and check
Condition~$(\star)$ for all these subsets. 
A very conservative estimate of necessary time
suggested that it would require thousands of years or more to go beyond $n=8$.
Copilot's idea was to look instead into a certain rooted graph, whose vertices
are constructed recursively level by level, and where the original 
sequence appears as the number of paths from vertices at level $n$
to the root. The total number of vertices at level $n$ seemed to be
growing much slower than the sequence itself, which suggested 
possibility for more efficient computations and, especially, for 
information storage. As a matter of fact,
Copilot observed that this number of vertices at level $n$, for small 
values of $n$, was given by the Fibonacci numbers with odd indices
and conjectured that this should hold in general.
With some effort, I managed to find a strategy to
prove this conjecture using an unusual recurrence 
for Fibonacci numbers, see Theorem~\ref{thm3}.
After this strategy was identified, 
Copilot was able to help with several steps of its execution.

When one cannot compute the elements of a sequence explicitly,
the question of growth of this sequence 
becomes especially relevant. Having computed the first $12$
terms of the sequence, it looked very plausible that
the growth of the sequence is exponential with some base
between $9$ and $11$. With substantial help of Copilot,
it turned out to be possible to find an exponential lower
bound with base $\approx 8.33$ and a factorial upper bound.
Here the contributions were distributed as follows:
\begin{itemize}
\item I looked for combinatorial lower or upper bounds which 
satisfy ``easy'' recurrence;
\item Copilot analyzes the recurrence using analytic tools
and establishes its growth rate (for example, in the case
of exponential growth this means: Copilot determines 
a polynomial relation satisfied by the generating function
and obtains the base of the exponential growth as the inverse of the
smallest positive root).
\end{itemize}
After a few refinements for the combinatorial lower or upper bounds,
I got convinced that $R_n$ can be defined recursively, but not in a
very trivial way. For combinatorial lower and upper bounds I used
statistics of the elements in $\mathcal{R}(n)$ with respect to the
number of points in the first column. Written properly, this
essentially gives the sequences $a(n,k)$, $b(n,k)$, $A(n)$
and $B(n)$ introduced above. The only problem was to guess the
correct recurrence. It took me quite a while to figure out
that the second term in \eqref{recurrence:b} should be a triple
and not a double sum. After that, it took many hours of computations
with the help of Copilot to correctly guess the summation
boundaries, with the final touch, namely $\min(m-1,k-1)$, in fact,
suggested by Copilot. After it was confirmed that the sequence
computes exactly what it should (which Copilot could compare
with the brute-forced combinatorial picture for small $n\leq 8$), 
to prove that this is the right thing was just some technical work.
The major consequence: it is now easy and fast to compute $R_n$ 
for the values of $n$ at least up to $100$. Copilot can
certainly do more as it has access to computational facilities
that significantly exceed anything I have access to.

With the exact recurrence at hand, it was not too difficult to
obtain significantly better lower and, especially, 
upper bounds for growth, where
the absolute majority of the heavy job was done by Copilot
(who apparently has much better knowledge of analytic
combinatorics than I have).

As was also suggested by Henning Krause, I looked a little bit
in the lattice structure of $\mathcal{R}(n)$. Here some basic
things, for example, description of atoms or join-irreducible
elements, were very easy. On the other hand, their duals
(coatoms and meet-irreducible elements) turned out to be
non-trivial. Here it was interesting to observe that Copilot
can easily compute small rank examples and even formulate
conjectures, but sometimes has serious difficulties in producing a
complete proof. Especially for the description of 
meet-irreducible elements, the original proof by Copilot contained
many gaps and shortcuts. It required a significant 
(combinatorial) effort from my side to make it work at the end.

\subsection*{Acknowledgements}
The author is partially supported by the Swedish Research Council.
Some computation in the paper were done using SageMath,
\cite{Sage}.

\subsection*{AI statement}

Many of the results and ideas in this paper arose 
from extensive exploratory discussions with Microsoft 
Copilot. Copilot’s role was really substantial 
at the exploratory and computational stages. 
It suggested computational strategies, helped identify 
patterns, assisted in formulating conjectural recurrences, 
and contributed to parts of the analytic-combinatorial 
analysis. These contributions are therefore explicitly 
acknowledged. However, all mathematical statements, proofs, 
corrections, and final formulations are the responsibility of the author.

\vspace{5mm}

\section{Representation theoretic background}\label{s2}

\subsection{Uniformly oriented quivers of type $A_n$}\label{s2.1}

Let $n$ be a positive integer. Consider the following quiver,
which we denote by $\mathtt{A}_n$:
\begin{displaymath}
\xymatrix{
1\ar[r]^{\alpha_1}&2\ar[r]^{\alpha_2}&
3\ar[r]^{\alpha_3}&\dots\ar[r]^{\alpha_{n-1}}&n
}
\end{displaymath}
Note that the underlying unoriented graph of $\mathtt{A}_n$
is a Dynkin diagram of type $A_n$.

\subsection{Representations}\label{s2.2}

Let us work over the field $\mathbb{C}$ of complex numbers.
A (finite dimensional) representation $V$ of $\mathtt{A}_n$ is an assignment
\begin{itemize}
\item of a finite dimensional (complex) vector space $V(i)$
to each vertex $i$ of $\mathtt{A}_n$;
\item and a linear operator $V(\alpha_i):V(i)\to V(i+1)$ to each 
arrow $\alpha_i$ of $\mathtt{A}_n$.
\end{itemize}
Given two representation $V$ and $V'$ of $\mathtt{A}_n$, a morphism
$f:V\to V'$ is an assignment of a linear map $f(i):V(i)\to V'(i)$, for
every vertex $i$, such that the following diagram commutes:
\begin{displaymath}
\xymatrix{
V(1)\ar[rr]^{V(\alpha_1)}\ar[d]_{f(1)}&&V(2)\ar[rr]^{V(\alpha_2)}\ar[d]_{f(2)}&&
V(3)\ar[rr]^{V(\alpha_3)}\ar[d]_{f(3)}&&
\dots\ar[rr]^{V(\alpha_{n-1})}\ar@{..}[d]&&V(n)\ar[d]_{f(n)}\\
V'(1)\ar[rr]^{V'(\alpha_1)}&&V'(2)\ar[rr]^{V'(\alpha_2)}&&
V'(3)\ar[rr]^{V'(\alpha_3)}&&\dots\ar[rr]^{V'(\alpha_{n-1})}&&V'(n)\\
}
\end{displaymath}
All (finite dimensional)
representations of $\mathtt{A}_n$ and their morphisms form a category,
denoted by $\mathtt{A}_n$-mod. The category $\mathtt{A}_n$-mod
is abelian, idempotent split, Krull-Schmidt, 
it has enough projectives and every object 
of this category has finite length.

\subsection{Indecomposable representations}\label{s2.3}

For $1\leq a\leq b\leq n$, denote by $\mathbf{V}_{a,b}$ the representation of
$\mathtt{A}_n$ defined as follows:
\begin{displaymath}
\mathbf{V}_{a,b}(i):=
\begin{cases}
\mathbb{C},& a\leq i\leq b;\\
0, & \text{else};
\end{cases}
\qquad
\mathbf{V}_{a,b}(\alpha_i):=
\begin{cases}
\mathrm{Id}_\mathbb{C},& a\leq i\leq b-1;\\
0, & \text{else}.
\end{cases}
\end{displaymath}
Gabriel's Theorem, see \cite{Ga}, says that the set
$\{\mathbf{V}_{a,b}\,:\, 1\leq a\leq b\leq n\}$
is a complete and irredundant list of representatives of the
isomorphism classes of indecomposable representations of 
$\mathtt{A}_n$. Consequently, any (finite dimensional)
representations of $\mathtt{A}_n$ can be written as a direct sum
of the $\mathbf{V}_{a,b}$'s uniquely, up to isomorphism.

\subsection{Extensions}\label{s2.4}

The following is easy to check, see, for example, \cite[Lemma~5]{KS}:
Unless $a< c\leq b+1$ and $b< d$, we have 
$\mathrm{Ext}^1(\mathbf{V}_{a,b},\mathbf{V}_{c,d})=0$.
If $a<c\leq b+1$ and $b< d$, 
then we have $\mathrm{Ext}^1(\mathbf{V}_{a,b},\mathbf{V}_{c,d})\cong\mathbb{C}$
and there is a non-split short exact sequence of the form
\begin{displaymath}
0\to  \mathbf{V}_{c,d}\to 
\mathbf{V}_{\min(a,c),\max(b,d)}\oplus
\mathbf{V}_{\max(a,c),\min(b,d)}
\to \mathbf{V}_{a,b} \to 0,
\end{displaymath}
where, by convention, $\mathbf{V}_{\max(a,c),\min(b,d)}=0$ provided that
$\max(a,c)>\min(b,d)$.

\subsection{Extension closed additive subcategories}\label{s2.5}

In \cite[Section~3]{KS}, among other things, a problem to compute the
number of full, additive, idempotent split subcategories of 
$\mathtt{A}_n$-mod that are closed under taking extensions is raised. 
Such a subcategory $\mathcal{X}$ is uniquely determined by the
indecomposable modules $\mathbf{V}_{a,b}$ that it contains.
By Subsection~\ref{s2.4}, being extension closed amounts to the following
condition:
if $\mathcal{X}$ contains $\mathbf{V}_{a,b}$ and $\mathbf{V}_{c,d}$,
for some $a< c\leq b+1$ and $b<d$, then 
$\mathcal{X}$ must contain $\mathbf{V}_{\min(a,c),\max(b,d)}$ as well as
$\mathbf{V}_{\max(a,c),\min(b,d)}=0$ (the latter, provided that
$\max(a,c)\leq \min(b,d)$). In \cite[Section~3]{KS}, we can find 
that the number of such subcategories for small values of $n$ is given by:
\begin{displaymath}
 2, 7, 34, 199, 1308, 9300.
\end{displaymath}
The two next terms $69978$ and $549559$ can be found in A393920, see \cite{OEIS},
computed by Fr{\'e}d{\'e}ric Chapoton.

\subsection{Connection}\label{s2.9}

For $1\leq a\leq b\leq n+1$, we associate the module $\mathbf{V}_{a,b}$
to the point $(n+1-b,a-1)\in \mathcal{Q}(n)$. This correspondence is bijective.
Given $1\leq a\leq b\leq n+1$ and $1\leq c\leq d\leq n+1$, the conditions
$a< c$, $c\leq b+1$ and $b<d$ are equivalent to
the conditions $a-1< c-1$, $(c-1)+(n+1-b)\leq n+1$ and  
$n+1-d< n+1-b$, respectively.
Finally, if these are satisfied,
the condition $\max(a,c)\leq \min(b,d)$ is equivalent to the
condition $\max(a-1,c-1)+\max(n+1-b,n+1-d)\leq n$.
Therefore $R_n$ counts the number of full, additive, idempotent split 
subcategories of  $\mathtt{A}_{n+1}$-mod that are closed under taking extensions.

\subsection{Conventions}\label{s2.95}

Let $\mathbf{2}^{\mathcal{Q}(n)}$ denote the boolean of $\mathcal{Q}(n)$.
For us, it will be convenient to 
identify a subset $Y\subset \mathcal{Q}(n)$ with the collection
$(Y_n,Y_{n-1},\dots,Y_0)$ of the first coordinates of its rows,
indexed by the second coordinate. This means that 
$Y_i$ is a subset of $\{0,1,\dots,n-i\}$ consisting of all those 
$x$ for which  $(x,i)\in Y$. For example, the subset
$\{(1,0),(2,0),(2,1),(0,3)\}$ of $\mathcal{Q}(3)$ is identified 
with $(\{0\},\varnothing,\{2\},\{1,2\})$.

We depict $Y\subset \mathcal{Q}(n)$ as a triangular
$0$-$1$ matrix in the obvious way. For example
$Y=(\{{\color{teal}0}\},\{{\color{orange}1}\},
\{{\color{magenta}0,1}\},\{{\color{violet}2,3}\})
\in \mathcal{Q}(3)$ will be depicted as follows:
\begin{displaymath}
\xymatrix@R=0.3mm@C=0.3mm{
{\color{teal}1}&&&\\
0&{\color{orange}1}&&\\
{\color{magenta}1}&{\color{magenta}1}&0&\\
0&0&{\color{violet}1}&{\color{violet}1}\\
}
\end{displaymath}
Here the color indicates which elements appear where in the picture.

\section{Proof of Main Result}\label{s3new}

\subsection{First column}\label{s3new.1}

Let $n\geq 0$. For $0\leq k \leq n+1$, consider the sets $\mathcal{R}(n)_k$
and $\mathcal{P}(n)_k$ consisting of all elements of 
$\mathcal{R}(n)$ and $\mathcal{P}(n)$, respectively, which 
contain exactly $k$ points of the form $(0,i)$.
Note that, geometrically, these points belong to the
first (leftmost) column of $\mathcal{Q}(n)$.
We set $r(n,k):=|\mathcal{R}(n)_k|$ and 
$p(n,k):=|\mathcal{P}(n)_k|$.
From
\begin{displaymath}
\mathcal{R}(n)=\coprod_{k=0}^{n+1}\mathcal{R}(n)_k
\quad\text{ and }\quad
\mathcal{P}(n)=\coprod_{k=0}^{n+1}\mathcal{P}(n)_k, 
\end{displaymath}
using the Rule of Sum, we have
\begin{equation}\label{rpvsRP}
R_n=\sum_{k=0}^{n+1} r(n,k) 
\quad\text{ and }\quad
P_n=\sum_{k=0}^{n+1} p(n,k). 
\end{equation}

Directly from the definitions, we have $p(n,0)=0$,
for all $n\geq 0$. Furthermore, for $n\geq 0$ and $k\not\in\{0,1,\dots,n+1\}$,
we have both  $p(n,k)=r(n,k)=0$. We also have
$p(0,1)=r(0,1)=r(0,0)=1$.

We will consider $r$ and $p$ as functions from $\mathbb{Z}^2$
to $\mathbb{Z}$ using the following conventions
(compare with Subsection~\ref{s1.2}):
\begin{itemize}
\item $r(n,i)=0$ if $i<0$ or $n<-2$;
\item $r(n,i)=0$ if $n<0$ and $i\neq 0$;
\item $r(-1,0)=r(-2,0)=1$;
\item $p(n,i)=0$ if $i\leq 0$ or $i>n+1$.
\end{itemize}

\subsection{Recurrence}\label{s3new.2}

\begin{proposition}\label{prop:rec1}
For $n>0$, we have
\begin{equation}\label{prop:rec1.1}
r(n,0)=\sum_{i=0}^{n}r(n-1,i).
\end{equation}
Moreover, for $n>0$ and $1\leq k\leq n+1$, we have
\begin{equation}\label{prop:rec1.2}
r(n,k)=p(n,k)+\sum_{j=0}^{n-1} p(j,k)\sum_{i=0}^{n-j-1}r(n-j-2,i).
\end{equation} 
\end{proposition}

\begin{proof}
For $Y\in \mathcal{R}(n)_0$, deleting the (empty) first column and shifting
one step to the left produces an element of $\mathcal{R}(n-1)$. Moreover,
this procedure is, clearly, reversible. This means that there is a bijection
between $\mathcal{R}(n)_0$ and $\mathcal{R}(n-1)$ which implies that
$|\mathcal{R}(n)_0|=|\mathcal{R}(n-1)|$. This translates into
$r(n,0)=R_{n-1}$ and hence \eqref{prop:rec1.1}
follows from \eqref{rpvsRP}.

To prove \eqref{prop:rec1.2}, we split $\mathcal{R}(n)_k$ 
according to the minimal second coordinate for an element 
in the first column. Note that, since $k>0$, this first column
is not empty. Let $(0,m)$, where $0\leq m\leq n$, be the 
element with the minimal second coordinate in the first column. 
Then Condition~$(\star)$ implies that the corresponding 
$Y\in \mathcal{R}(n)$ must have the form
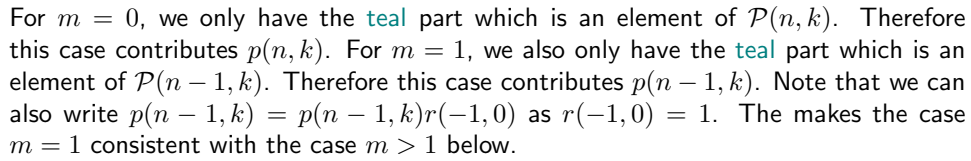
\begin{displaymath}
\xymatrix@R=0.3mm@C=0.3mm{
{\color{teal}\bullet}&&&&&&&&&\\
{\color{teal}\bullet}&{\color{teal}\bullet}&&&&&&&&\\
{\color{teal}1}&{\color{teal}\bullet}&{\color{teal}\bullet}&&&&&&&&\\
0&0&0&0&&&&&&\\
0&0&0&0&{\color{violet}\bullet}&&&&&\\
0&0&0&0&
{\color{violet}\bullet}&{\color{violet}\bullet}&&&&\\
0&0&0&0&{\color{violet}\bullet}&{\color{violet}\bullet}&{\color{violet}\bullet}&&&\\
0&0&0&0&
{\color{violet}\bullet}&{\color{violet}\bullet}&{\color{violet}\bullet}&
{\color{violet}\bullet}&&\\
0&0&0&0&{\color{violet}\bullet}&{\color{violet}\bullet}&
{\color{violet}\bullet}&{\color{violet}\bullet}&{\color{violet}\bullet}&
}
\end{displaymath}
For $m=0$, we only have the {\color{teal}teal} part which is an element
of $\mathcal{P}(n,k)$. Therefore this case contributes $p(n,k)$.
For $m=1$, we also only have the {\color{teal}teal} part which is an element
of $\mathcal{P}(n-1,k)$. Therefore this case contributes $p(n-1,k)$.
Note that we can also write $p(n-1,k)=p(n-1,k)r(-1,0)$ as $r(-1,0)=1$.
The makes the case $m=1$ consistent with the case $m>1$ below.

If $m>1$, the {\color{teal}teal}  and the {\color{violet}violet} parts 
do not interact with respect to Condition~$(\star)$.
Therefore the {\color{teal}teal} part is in bijection with 
$\mathcal{P}(n-m,k)$ and the  {\color{violet}violet} part is
independently in
bijection with $\mathcal{R}(m-2)$. By the Rule of Product, the
total contribution is $p(n-m,k)R_{m-2}$. 

Finally, setting $j=n-m$, using \eqref{rpvsRP}, 
and adding up over all $j$, we get \eqref{prop:rec1.2}.
\end{proof}

The following proposition is the crucial observation for
the proof of the main result.

\begin{proposition}\label{prop:rec2}
For $n>0$ and $1\leq k\leq n+1$, we have
\begin{equation}\label{prop:rec2.1}
p(n,k)=r(n-1,k-1)+
\sum_{m=1}^{n}\sum_{r=0}^{\min(m-1,k-1)} 
r(m-2,r)\sum_{q=k-1-r}^{n-m+1} p(n-m,q).
\end{equation} 
\end{proposition}

\begin{proof}
The elements of $\mathcal{Q}(n)$ of the form $(i,0)$
form what we will call the {\em bottom row}.
We split $\mathcal{P}(n)_k$ according to the second 
minimal first coordinate for an element in this bottom row.
If the bottom row is a singleton, deleting it gives an element in 
$\mathcal{R}(n-1)_{k-1}$ which contributes $r(n-1,k-1)$, the first
summand of \eqref{prop:rec2.1}.

If the bottom row is not a singleton, let $(m,0)$ be the 
element with the second minimal first coordinate in this bottom row. 
Then Condition~$(\star)$ implies that the corresponding 
$Y\in \mathcal{P}(n)$ must have the form
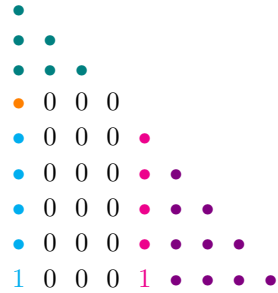
\begin{displaymath}
\xymatrix@R=0.3mm@C=0.3mm{
{\color{teal}\bullet}&&&&&&&&&\\
{\color{teal}\bullet}&{\color{teal}\bullet}&&&&&&&&\\
{\color{teal}\bullet}&{\color{teal}\bullet}&{\color{teal}\bullet}&&&&&&&&\\
{\color{orange}\bullet}&0&0&0&&&&&&\\
{\color{cyan}\bullet}&0&0&0&{\color{magenta}\bullet}&&&&&\\
{\color{cyan}\bullet}&0&0&0&
{\color{magenta}\bullet}&{\color{violet}\bullet}&&&&\\
{\color{cyan}\bullet}&0&0&0&
{\color{magenta}\bullet}&{\color{violet}\bullet}&{\color{violet}\bullet}&&&\\
{\color{cyan}\bullet}&0&0&0&
{\color{magenta}\bullet}&{\color{violet}\bullet}&{\color{violet}\bullet}&
{\color{violet}\bullet}&&\\
{\color{cyan}1}&0&0&0&{\color{magenta}1}&{\color{violet}\bullet}&
{\color{violet}\bullet}&{\color{violet}\bullet}&{\color{violet}\bullet}&
}
\end{displaymath}
Here the {\color{teal}teal}  part is 
an element of $\mathcal{R}(m-2,r)$, for some
$0\leq r\leq\min(m-1,k-1)$, and hence can be chosen
in $r(m-2,r)$ different ways under these restrictions on $r$.

The union of the {\color{violet}violet}
and {\color{magenta}magenta} parts is an element of
$\mathcal{P}(n-m)$. The union of the 
{\color{violet}violet}, {\color{magenta}magenta},
{\color{cyan}cyan} and {\color{orange}orange} parts is 
an element of $\mathcal{P}(n-m+1)$.  Unfortunately,
we see that this element of $\mathcal{P}(n-m+1)$
is not arbitrary as its bottom row starts with two $1$'s
and also its first column should contain exactly
$k-r$ elements.
However, we can in any case say that
Condition~$(\star)$ cannot be triggered by an interaction of
the {\color{teal}teal} part with any of the remaining four
colored parts (that is, if we ignore the black $0$'s).
In particular, the choice of an appropriate 
element of $\mathcal{P}(n-m+1)$ for the four-color part
is independent from the {\color{teal}teal} part.

It remains to show that the set of all elements of 
$\mathcal{P}(n-m+1)$ whose bottom row starts with two $1$'s
and which have exactly $k-r$ elements in the first column
has cardinality
\begin{equation}\label{eq:rewrite}
\sum_{q=k-1-r}^{n-m+1} p(n-m,q).
\end{equation}
If we can prove this, \eqref{prop:rec2.1} follows by 
summing up everything above and applying the
Rules of Sum and Product.

Let us take a closer look at the {\color{magenta}magenta}
and {\color{cyan}cyan} columns. Due to the $1$'s 
in the origin and at position $(m,0)$,
any $1$ in the {\color{cyan}cyan} column yields, by
Condition~$(\star)$, a $1$ at the same 
height in the {\color{magenta}magenta}
column. Moreover, if we have $1$'s at heights $i<j$
and the {\color{magenta}magenta}
column and $1$ at height $j$ in {\color{cyan}cyan} column,
Condition~$(\star)$ forces $1$ at height $i$
in the {\color{cyan}cyan} column. In other words, 
if we read the heights for positions of $1$'s in the 
{\color{cyan}cyan} column in the natural order, 
we get a prefix of the similar  word for the {\color{magenta}magenta}
column. We refer to this as the {\em prefix property}. 
Finally, if the {\color{cyan}cyan}
and {\color{magenta}magenta} columns coincide, 
the {\color{orange}orange} element can be made equal to $1$.
If the {\color{cyan}cyan} and {\color{magenta}magenta} 
columns are different,  the {\color{orange}orange} element 
must be $0$.

Conversely, pick any element of $\mathcal{P}(n-m)$
for the union of  {\color{violet}violet} and {\color{magenta}magenta} 
parts. Now choose as the {\color{cyan}cyan} column any column with the
prefix property with respect to the {\color{magenta}magenta} 
column. Set  the {\color{orange}orange} element 
to be $0$. We obtain an element in $\mathcal{P}(n-m+1)$.
Additionally, if, as the {\color{cyan}cyan} column, we choose
the {\color{magenta}magenta} column, we can also make
the {\color{orange}orange} element $1$ and we still 
obtain an element in $\mathcal{P}(n-m+1)$.

Now let us remember that in our element in $\mathcal{P}(n-m+1)$
we need to have exactly $k-r$ points in the first column
at the end of the day. Therefore we need to start either with an 
element of $\mathcal{P}(n-m)$ for which the 
{\color{magenta}magenta} column has at least $k-r$ points to start
with (and we take as the {\color{cyan}cyan} column 
the unique prefix with exactly $k-r$ points and
set the {\color{orange}orange} element to be $0$), or
we start with an element in which the
{\color{magenta}magenta} column has $k-r-1$ points,
we copy it to the {\color{cyan}cyan} place and
we set the {\color{orange}orange} to be $1$.
Using the Rule of Sum, gives us \eqref{eq:rewrite} and completes the proof.
\end{proof}

\subsection{Proof of Theorem A}\label{s3new.3}

By definition, the initial values of $r(n,k)$
coincide with those of $a(n,k)$ and the initial values
of $p(n,k)$ coincide with those of $b(n,k)$.
Comparing \eqref{recurrence:a1},
\eqref{recurrence:a2} and \eqref{recurrence:b}
with \eqref{prop:rec1.1}, \eqref{prop:rec1.2}
and \eqref{prop:rec2.1}, respectively, we see that
the pair $r(n,k)$ and $p(n,k)$ satisfies the same recurrence
as the pair $a(n,k)$ and $b(n,k)$. Consequently,
$a(n,k)=r(n,k)$ and $b(n,k)=p(n,k)$, for all $n$ and $k$.

The claim of Theorem A now follows by comparing the definition
\eqref{abvsAB} with \eqref{rpvsRP}.\hfill\qed

\section{Computations}\label{s5}

Microsoft Copilot wrote an executable SageMath code (see Appendix)
to compute $a$, $b$, $A$ and $B$. Below are some results of this
computation.

Values of $a(n,k)$, up to $n=10$:

\resizebox{\textwidth}{!}{
$
\begin{array}{ll}
 a(0,\cdot) &= (1,1),\\
 a(1,\cdot) &= (2,3,2),\\
 a(2,\cdot) &= (7,11,11,5),\\
 a(3,\cdot) &= (34,52,58,41,14),\\
 a(4,\cdot) &= (199,295,338,280,154,42),\\
 a(5,\cdot) &= (1308,1891,2174,1925,1288,582,132),\\
 a(6,\cdot) &= (9300,13188,15113,13805,10178,5754,2211,429),\\
 a(7,\cdot) &= (69978,97752,111480,103460,80486,51324,25212,8437,1430),\\
 a(8,\cdot) &= (549559,758559,860838,805848,647528,443688,250404,108966,32318,4862),\\
 a(9,\cdot) &= (4462570,6100345,6891414,6483165,5319384,3817464,2349600,1192191,466180,124202,16796),\\
 a(10,\cdot) &= (37223311,50479363,56793341,53584575,44595474,33030522,21543918,12068628,5569850,1978834,478686,58786).
\end{array}
$
}

Values of $b(n,k)$, up to $n=10$:

\resizebox{\textwidth}{!}{
$
\begin{array}{ll}
 b(0,\cdot) &= (0,1),\\
 b(1,\cdot) &= (0,2,2),\\
 b(2,\cdot) &= (0,7,9,5),\\
 b(3,\cdot) &= (0,34,45,36,14),\\
 b(4,\cdot) &= (0,199,261,234,140,42),\\
 b(5,\cdot) &= (0,1308,1692,1584,1120,540,132),\\
 b(6,\cdot) &= (0,9300,11880,11331,8680,5130,2079,429),\\
 b(7,\cdot) &= (0,69978,88452,85104,68110,44820,22869,8008,1430),\\
 b(8,\cdot) &= (0,549559,688581,665334,546672,383400,222453,100100,30888,4862),\\
 b(9,\cdot) &= (0,4462570,5550786,5374296,4491480,3281796,2062368,1073072,432432,119340,16796),\\
 b(10,\cdot) &= (0,37223311,46016793,44592651,37702560,28331370,18779607,10724714,5065632,1849770,461890,58786).
\end{array}
$
}

The values of $A(n)$ and $B(n)$, up to $n=20$, can be found in
Figure~\ref{fig22}.

\begin{figure}
\resizebox{!}{3cm}{
$
\begin{array}{r|r|r}
 n & A(n) & B(n)\\
\hline
 0 & 2 & 1\\
 1 & 7 & 4\\
 2 & 34 & 21\\
 3 & 199 & 129\\
 4 & 1308 & 876\\
 5 & 9300 & 6376\\
 6 & 69978 & 48829\\
 7 & 549559 & 388771\\
 8 & 4462570 & 3191849\\
 9 & 37223311 & 26864936\\
 10 & 317405288 & 230807084\\
 11 & 2756819108 & 2017470636\\
 12 & 24321036896 & 17895818664\\
 13 & 217459259352 & 160769210256\\
 14 & 1967105229930 & 1460332969869\\
 15 & 17976521266263 & 13394243081415\\
 16 & 165766191111606 & 123914091078435\\
 17 & 1540875893526165 & 1155197707111140\\
 18 & 14426380369147026 & 10843858784096205\\
 19 & 135942674690757903 & 102426415090389825\\
 20 & 1288538824523879508 & 972948700592516736
\end{array}
$
}
\caption{Values of $A(n)$ and $B(n)$, up to $n=20$}
\label{fig22}
\end{figure}

\section{States and state graph}\label{s3}

\subsection{States}\label{s3.3}

Define the map $\boldsymbol{\pi}_n:\mathbf{2}^{\mathcal{Q}(n+1)}\to
\mathbf{2}^{\mathcal{Q}(n)}$ as follows:
\begin{displaymath}
\boldsymbol{\pi}_n\big((Y_n,Y_{n-1},\dots,Y_1,Y_0)\big)=
(Y_n,Y_{n-1},\dots,Y_1).
\end{displaymath}
Note that $\boldsymbol{\pi}_n$ is, obviously, surjective. 
The map $\boldsymbol{\pi}_n$ restricts to a surjective map
$\boldsymbol{\pi}_n:\mathcal{R}(n+1)\to \mathcal{R}(n)$.

Given $Y\in \mathcal{R}(n)$, we define the {\em state 
$\mathbf{S}(Y)$ of $Y$} as the set of all possible subsets 
$A\subset \{0,1,\dots,n+1\}$ such that
$(Y_n,Y_{n-1},\dots,Y_0,A)\in \mathcal{R}(n+1)$.
The essential information contained in $\mathbf{S}(Y)$
is that of the preimage of $Y$ under $\boldsymbol{\pi}_n$. In more detail,
for $Y\in \mathcal{R}(n)$, mapping
$A\in \mathbf{S}(Y)$ to $(Y_n,Y_{n-1},\dots,Y_0,A)$
defines a bijection between $\mathbf{S}(Y)$
and $\boldsymbol{\pi}_n^{-1}(Y)$. For example, we have
\begin{displaymath}
\mathbf{S}\big((\varnothing)\big)=
\{\varnothing,\{0\},\{1\},\{0,1\}\}
\end{displaymath}
while
\begin{displaymath}
\mathbf{S}\big((\{0\})\big)=
\{\varnothing,\{0\},\{0,1\}\}.
\end{displaymath}
The latter is due to the fact that $(\{0\},\{1\})\not\in\mathcal{R}(1)$.

It will be convenient to set $\mathcal{Q}(-1):=\varnothing$,
so that $\mathbf{2}^{\mathcal{Q}(-1)}=\mathcal{R}(-1)=\{\varnothing\}$.
Also, we set $\mathbf{S}\big(\varnothing\big)=\{\varnothing,\{0\}\}$
(note the difference between $\mathbf{S}\big((\varnothing)\big)$ 
and $\mathbf{S}\big(\varnothing\big)$, where 
we have $(\varnothing)\in \mathcal{R}(0)$
while $\varnothing\in \mathcal{R}(-1)$).
Define $\boldsymbol{\pi}_{-1}$ from
$\mathbf{2}^{\mathcal{Q}(0)}=\mathcal{R}(0)$ to 
$\mathbf{2}^{\mathcal{Q}(-1)}=\mathcal{R}(-1)$
in the only possible way, that is, sending each element of 
$\mathbf{2}^{\mathcal{Q}(0)}$ to the only element of $\mathbf{2}^{\mathcal{Q}(-1)}$.

\subsection{State graph}\label{s3.4}

Consider the infinite edge labeled directed 
graph $\Psi$ defined as follows: 
\begin{itemize}
\item the set of vertices of $\Psi$ is 
$\displaystyle \coprod_{i\geq -1}\mathcal{R}(i)$;
\item for a vertex $v\in \mathcal{R}(i)$, where $i\geq 0$, we have one
oriented edge from  $v$ to  $\boldsymbol{\pi}_{i-1}(v)$
and this edge is labeled by $V_0$, provided that 
$v=(V_i,V_{i-1},\dots,V_0)$.
\end{itemize}
The graph $\Psi$ is a rooted tree with root 
$\{\varnothing\}\in \mathcal{R}(-1)$. Also, for $i\geq -1$,
the elements of  $\mathcal{R}(i)$ are exactly the elements which have 
distance $i+1$ from the root.

Now we want to define a quotient of $\Psi$: 
for $i\geq -1$ and $v,w\in \mathcal{R}(i)$, we 
say that $v\sim_i w$ provided that the set of labels of the 
edges pointing towards $v$ coincides with 
the set of labels of the  edges pointing towards $w$.
More precisely, $v\sim_i w$, provided that 
\begin{displaymath}
\bigcup_{\overset{u=(U_{i+1},U_{i},\dots,U_0)\in \mathcal{R}(i+1)}
{\boldsymbol{\pi}_{i}(u)=v}}\{ U_0\} \quad=\quad
\bigcup_{\overset{u=(U_{i+1},U_{i},\dots,U_0)\in \mathcal{R}(i+1)}
{\boldsymbol{\pi}_{i}(u)=w}}\{ U_0\}.
\end{displaymath}
Clearly, $\sim_i$ is an equivalence relation on $\mathcal{R}(i)$.
For $v\in \mathcal{R}(i)$, we will denote the $\sim_i$-equivalence
class containing $v$ by $v_{\sim_i}$.
Define the quotient $\overline{\Psi}$ as follows:
\begin{itemize}
\item the set of vertices of $\Psi$ is  
$\displaystyle \coprod_{i\geq -1}\left(\mathcal{R}(i)/\sim_i\right)$;
\item for $v\in \mathcal{R}(i)$ and $w\in \mathcal{R}(i-1)$, where $i\geq 0$, 
there is oriented edge from $v_{\sim_i}$ to $w_{\sim_{i-1}}$ if and only
if there exists $v'\in v_{\sim_i}$ and $w'\in w_{\sim_{i-1}}$
such that there is an oriented edge from $v'$ to $w'$
in $\Psi$. This edge from $v_{\sim_i}$ to $w_{\sim_{i-1}}$ has the same
label as the edge from  $v'$ to $w'$.
\end{itemize}

For a positive integer $n$, set $\underline{n}=\{1,2,\dots,n\}$.
We identify the subsets of $\underline{n}$ with $0$-$1$ sequences
of length $n$ in the usual way: a $0$-$1$ sequence
$\mathbf{a}:=a_1a_2\dots a_n$ corresponds to its support, that is, 
the subset $\mathrm{supp}(\mathbf{a})=\{i\,:\, a_i=1\}$ of $\underline{n}$.

Now we define a new infinite graph $\Sigma$ in a recursive way.
This graph will have the following properties:
\begin{itemize}
\item the set of vertices of $\Sigma$ is a disjoint union of the sets
$\Sigma_i$, for $i\geq -1$;
\item the elements of $\Sigma_i$ will be collections (sets) of subsets of 
$\underline{i+2}$, where each such subset is interpreted as a $0$-$1$ sequence
of length $i+2$;
\item oriented edges in $\Sigma$ will only exist from 
elements of $\Sigma_i$ to elements of $\Sigma_{i-1}$,
for $i\geq 0$.
\end{itemize}
We start with the basis of our recurrence:
\begin{itemize}
\item the set $\Sigma_{-1}$ is a singleton, namely, the only element
in $\Sigma_{-1}$ is the set $\{0,1\}$.
\end{itemize}
Now, assume that $\Sigma_i$ is constructed. Then 
we define $\Sigma_{i+1}$ and the edges
from $\Sigma_{i+1}$ to $\Sigma_i$ as follows: 
let $N\in \Sigma_i$ and  $\mathbf{a}=a_1a_2\dots a_{i+2}\in N$.
Let $Y\subset \underline{i+2}$ be the subset 
corresponding to $\mathbf{a}$. Then the following set $M$ belongs to
$\Sigma_{i+1}$ and $\Sigma$ has an edge from $M$ to $N$
labeled by $Y$ (alternatively, by $\mathbf{a}$): the set $M$ consists of  all
$0$-$1$ sequences $\mathbf{b}=b_1b_2\dots b_{n+3}$ which satisfy the following
three conditions:
\begin{enumerate}[$($A$)$]
\item\label{condA} if $b_{n+3}=1$, then, 
$\mathrm{supp}(\mathbf{a})\subset \mathrm{supp}(\mathbf{b})$;
\item\label{condB} if $b_s=1$, for some $s<n+3$, and 
$a_t=1$, for some $t<s$, then $b_t=a_s=1$;
\item\label{condC} $b_1b_2\dots b_{n+2}\in N$.
\end{enumerate}
The part of $\Sigma$ corresponding to the union of the 
$\Sigma_i$, for $i=-1,0,2,3$, is shown in Figure~\ref{fig1}.
We note that the idea to consider $\Sigma$ is due to 
Microsoft Copilot.

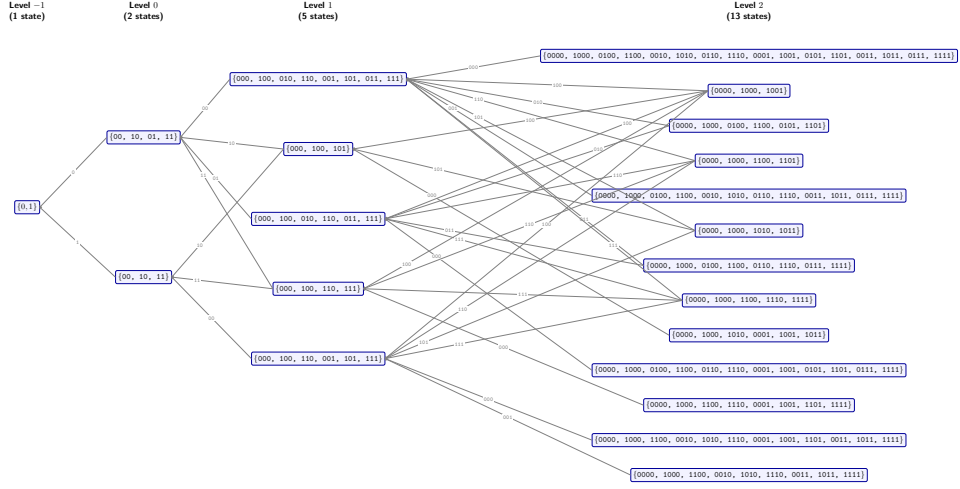
\begin{figure}
\resizebox{\textwidth}{!}{
\begin{tikzpicture}[
  statebox/.style={draw=blue!60!black, rounded corners=2pt, fill=blue!5, align=center, inner sep=3pt},
  edgelabel/.style={fill=white, inner sep=1pt, font=\scriptsize},
  levellabel/.style={font=\bfseries\large, align=center}
]

\node[levellabel] at (0.00,15.4) {Level $-1$\\(1 state)};
\node[levellabel] at (5.00,15.4) {Level $0$\\(2 states)};
\node[levellabel] at (12.50,15.4) {Level $1$\\(5 states)};
\node[levellabel] at (31.00,15.4) {Level $2$\\(13 states)};

\node[statebox] (Nm1x0) at (0.00,7.00) {$\{0,1\}$};
\node[statebox] (N0x0) at (5.00,10.00) {{\ttfamily \{00, 10, 01, 11\}}};
\node[statebox] (N0x1) at (5.00,4.00) {{\ttfamily \{00, 10, 11\}}};
\node[statebox] (N1x0) at (12.50,12.50) {{\ttfamily \{000, 100, 010, 110, 001, 101, 011, 111\}}};
\node[statebox] (N1x1) at (12.50,9.50) {{\ttfamily \{000, 100, 101\}}};
\node[statebox] (N1x2) at (12.50,6.50) {{\ttfamily \{000, 100, 010, 110, 011, 111\}}};
\node[statebox] (N1x3) at (12.50,3.50) {{\ttfamily \{000, 100, 110, 111\}}};
\node[statebox] (N1x4) at (12.50,0.50) {{\ttfamily \{000, 100, 110, 001, 101, 111\}}};
\node[statebox] (N2x0) at (31.00,13.50) {{\ttfamily \{0000, 1000, 0100, 1100, 0010, 1010, 0110, 1110, 0001, 1001, 0101, 1101, 0011, 1011, 0111, 1111\}}};
\node[statebox] (N2x1) at (31.00,12.00) {{\ttfamily \{0000, 1000, 1001\}}};
\node[statebox] (N2x2) at (31.00,10.50) {{\ttfamily \{0000, 1000, 0100, 1100, 0101, 1101\}}};
\node[statebox] (N2x3) at (31.00,9.00) {{\ttfamily \{0000, 1000, 1100, 1101\}}};
\node[statebox] (N2x4) at (31.00,7.50) {{\ttfamily \{0000, 1000, 0100, 1100, 0010, 1010, 0110, 1110, 0011, 1011, 0111, 1111\}}};
\node[statebox] (N2x5) at (31.00,6.00) {{\ttfamily \{0000, 1000, 1010, 1011\}}};
\node[statebox] (N2x6) at (31.00,4.50) {{\ttfamily \{0000, 1000, 0100, 1100, 0110, 1110, 0111, 1111\}}};
\node[statebox] (N2x7) at (31.00,3.00) {{\ttfamily \{0000, 1000, 1100, 1110, 1111\}}};
\node[statebox] (N2x8) at (31.00,1.50) {{\ttfamily \{0000, 1000, 1010, 0001, 1001, 1011\}}};
\node[statebox] (N2x9) at (31.00,0.00) {{\ttfamily \{0000, 1000, 0100, 1100, 0110, 1110, 0001, 1001, 0101, 1101, 0111, 1111\}}};
\node[statebox] (N2x10) at (31.00,-1.50) {{\ttfamily \{0000, 1000, 1100, 1110, 0001, 1001, 1101, 1111\}}};
\node[statebox] (N2x11) at (31.00,-3.00) {{\ttfamily \{0000, 1000, 1100, 0010, 1010, 1110, 0001, 1001, 1101, 0011, 1011, 1111\}}};
\node[statebox] (N2x12) at (31.00,-4.50) {{\ttfamily \{0000, 1000, 1100, 0010, 1010, 1110, 0011, 1011, 1111\}}};

\draw[<-, gray] (Nm1x0.east) -- (N0x0.west) node[midway, edgelabel] {{\ttfamily 0}};
\draw[<-, gray] (Nm1x0.east) -- (N0x1.west) node[midway, edgelabel] {{\ttfamily 1}};
\draw[<-, gray] (N0x0.east) -- (N1x0.west) node[midway, edgelabel] {{\ttfamily 00}};
\draw[<-, gray] (N0x0.east) -- (N1x1.west) node[midway, edgelabel] {{\ttfamily 10}};
\draw[<-, gray] (N0x0.east) -- (N1x2.west) node[midway, edgelabel] {{\ttfamily 01}};
\draw[<-, gray] (N0x0.east) -- (N1x3.west) node[near start, edgelabel] {{\ttfamily 11}};
\draw[<-, gray] (N0x1.east) -- (N1x1.west) node[near start, edgelabel] {{\ttfamily 10}};
\draw[<-, gray] (N0x1.east) -- (N1x3.west) node[near start, edgelabel] {{\ttfamily 11}};
\draw[<-, gray] (N0x1.east) -- (N1x4.west) node[midway, edgelabel] {{\ttfamily 00}};
\draw[<-, gray] (N1x0.east) -- (N2x0.west) node[midway, edgelabel] {{\ttfamily 000}};
\draw[<-, gray] (N1x0.east) -- (N2x1.west) node[midway, edgelabel] {{\ttfamily 100}};
\draw[<-, gray] (N1x0.east) -- (N2x2.west) node[midway, edgelabel] {{\ttfamily 010}};
\draw[<-, gray] (N1x0.east) -- (N2x3.west) node[near start, edgelabel] {{\ttfamily 110}};
\draw[<-, gray] (N1x0.east) -- (N2x4.west) node[near start, edgelabel] {{\ttfamily 001}};
\draw[<-, gray] (N1x0.east) -- (N2x5.west) node[near start, edgelabel] {{\ttfamily 101}};
\draw[<-, gray] (N1x0.east) -- (N2x6.west) node[near end, edgelabel] {{\ttfamily 011}};
\draw[<-, gray] (N1x0.east) -- (N2x7.west) node[near end, edgelabel] {{\ttfamily 111}};
\draw[<-, gray] (N1x1.east) -- (N2x1.west) node[midway, edgelabel] {{\ttfamily 100}};
\draw[<-, gray] (N1x1.east) -- (N2x5.west) node[near start, edgelabel] {{\ttfamily 101}};
\draw[<-, gray] (N1x1.east) -- (N2x8.west) node[near start, edgelabel] {{\ttfamily 000}};
\draw[<-, gray] (N1x2.east) -- (N2x1.west) node[near end, edgelabel] {{\ttfamily 100}};
\draw[<-, gray] (N1x2.east) -- (N2x2.west) node[near end, edgelabel] {{\ttfamily 010}};
\draw[<-, gray] (N1x2.east) -- (N2x3.west) node[near end, edgelabel] {{\ttfamily 110}};
\draw[<-, gray] (N1x2.east) -- (N2x6.west) node[near start, edgelabel] {{\ttfamily 011}};
\draw[<-, gray] (N1x2.east) -- (N2x7.west) node[near start, edgelabel] {{\ttfamily 111}};
\draw[<-, gray] (N1x2.east) -- (N2x9.west) node[near start, edgelabel] {{\ttfamily 000}};
\draw[<-, gray] (N1x3.east) -- (N2x1.west) node[very near start, edgelabel] {{\ttfamily 100}};
\draw[<-, gray] (N1x3.east) -- (N2x3.west) node[midway, edgelabel] {{\ttfamily 110}};
\draw[<-, gray] (N1x3.east) -- (N2x7.west) node[midway, edgelabel] {{\ttfamily 111}};
\draw[<-, gray] (N1x3.east) -- (N2x10.west) node[midway, edgelabel] {{\ttfamily 000}};
\draw[<-, gray] (N1x4.east) -- (N2x1.west) node[midway, edgelabel] {{\ttfamily 100}};
\draw[<-, gray] (N1x4.east) -- (N2x3.west) node[near start, edgelabel] {{\ttfamily 110}};
\draw[<-, gray] (N1x4.east) -- (N2x5.west) node[very near start, edgelabel] {{\ttfamily 101}};
\draw[<-, gray] (N1x4.east) -- (N2x7.west) node[near start, edgelabel] {{\ttfamily 111}};
\draw[<-, gray] (N1x4.east) -- (N2x11.west) node[midway, edgelabel] {{\ttfamily 000}};
\draw[<-, gray] (N1x4.east) -- (N2x12.west) node[midway, edgelabel] {{\ttfamily 001}};
\end{tikzpicture}
}
\caption{The initial part of the graph $\Sigma$}
\label{fig1}
\end{figure}

We aim to prove the following:

\begin{theorem}\label{thm1}
Sending $Y\in\mathcal{R}(i)$ to $\mathbf{S}(Y)$, for $i\geq -1$,
gives rise to an isomorphism between the labeled oriented
graphs $\overline{\Psi}$ and $\Sigma$.
\end{theorem}

To prove Theorem~\ref{thm1}, we need some preparation.

\begin{lemma}\label{lem1}
For any $i\geq -1$ and any $Y\in \mathcal{R}(i)$, we have
$\mathbf{S}(Y)\in \Sigma_i$.
\end{lemma}

\begin{proof}
We use induction on $i$. The basis $i=-1$ follows directly from the
definitions. To prove the induction step, let 
$Y\in \mathcal{R}(i)$, for some $i>-1$. Then
$\boldsymbol{\pi}_{i-1}(Y)\in \mathcal{R}(i)$ and hence,
by induction, we have 
$\mathbf{S}(\boldsymbol{\pi}_{i-1}(Y))\in \Sigma_{i-1}$.
If $Y=(Y_i,Y_{i-1},\dots,Y_0)$, then, by definition,
$\boldsymbol{\pi}_{i-1}(Y)=(Y_i,Y_{i-1},\dots,Y_1)$ and 
$Y_0\in \mathbf{S}(\boldsymbol{\pi}_{i-1}(Y))$. By our recursive
construction of $\Sigma$, the set $\Sigma_i$ contains 
an element $N$ and an edge from $N$ to $\mathbf{S}(\boldsymbol{\pi}_{i-1}(Y))$
labeled by $Y_0$. This $N$ consists of all $0$-$1$ strings of length
$i+2$ which extend the strings from $\mathbf{S}(\boldsymbol{\pi}_{i-1}(Y))$
by adding on symbol on the right and are compatible with $Y_0$ in the
sense of Conditions~\eqref{condA} and \eqref{condB}. We claim that
$N=\mathbf{S}(Y)$, which implies the claim of the lemma.

To prove that $N=\mathbf{S}(Y)$, let $X$ be a subset of $\{0,1,\dots,i+1\}$ 
such that we have
$(Y_i,Y_{i-1},\dots,Y_0,X)\in \mathcal{R}(i+1)$. Note that
Condition~$(\star)$ applied to any $Y_s$ and $Y_t$ is automatic as 
$Y\in \mathcal{R}(i)$. Further, Condition~$(\star)$
applied to $Y_0$ and $X$ translates exactly into the compatibility described in
Conditions~\eqref{condA} and \eqref{condB}.
And, finally, Condition~$(\star)$ applied to any $Y_s$, for $s>0$,
and $X$ says exactly that, erasing the rightmost letter from $X$,
we get some element in $\mathbf{S}(\boldsymbol{\pi}_{i-1}(Y))$, 
which should be compared with Condition~\eqref{condC}.
Now $N=\mathbf{S}(Y)$ follows from the definitions.
\end{proof}

\begin{lemma}\label{lem2}
The set $\Sigma_i$ coincides with the set
$\{\mathbf{S}(Y)\,:\, Y\in \mathcal{R}(i)\}$.
\end{lemma}

\begin{proof}
From Lemma~\ref{lem1}, we have 
$\{\mathbf{S}(Y)\,:\, Y\in \mathcal{R}(i)\}\subset \Sigma_i$.
We now prove that the set
$\Sigma_i\setminus \{\mathbf{S}(Y)\,:\, Y\in \mathcal{R}(i)\}$
is empty, by induction on $i$. The basis of the induction is clear.

To prove the induction step, let $N\in \Sigma_i$, for some  $i\geq 0$. 
Then, by construction of $\Sigma$, there is $M\in \Sigma_{i-1}$ and 
$Q\in M$ such that there is an oriented edge from $N$ to $M$
in $\Sigma$ labeled by $Q$. By induction, we know that 
$M=\mathbf{S}(X)$, for some element
$X=(X_{i-1},X_{i-2},\dots,X_0)\in \mathcal{R}(i-1)$. 
Since $Q\in\mathbf{S}(X)$, we have 
$(X_{i-1},X_{i-2},\dots,X_0,Q)\in \mathcal{R}(i)$.
We claim that $N=\mathbf{S}((X_{i-1},X_{i-2},\dots,X_0,Q))$.
Verification of this amounts to comparing the definition of 
$N$ with Condition~$(\star)$ similarly to the last 
paragraph in the proof of Lemma~\ref{lem1}.
\end{proof}

Due to Lemma~\ref{lem2}, the graph $\Sigma$ is called the {\em state graph}.

\begin{proof}[Proof of Theorem~\ref{thm1}]
By Lemma~\ref{lem1}, sending $Y\in\mathcal{R}(i)$ to $\mathbf{S}(Y)$, 
for $i\geq -1$, gives rise to a well-defined map $\mathbf{f}$ from the set of 
vertices of $\Psi$ to the set of vertices of $\Sigma$.
Given an edge from some $Y\in\mathcal{R}(i)$ to some $X\in\mathcal{R}(i-1)$
in $\Psi$, this edge is labeled by $Y_0$, if $Y=(Y_i,Y_{i-1},\dots,Y_0)$
and $X=(Y_i,Y_{i-1},\dots,Y_1)$. We therefore have $Y_0\in\mathbf{S}(X)$
and from the proof of Lemma~\ref{lem2} we deduce that there is 
an edge from $\mathbf{S}(Y)$  to $\mathbf{S}(X)$ in 
$\Sigma$ labeled by $Y_0$. Therefore $\mathbf{f}$ is 
a graph homomorphism from $\Psi$ to $\Sigma$ which is surjective
both at the level of vertices and at the level of edges. 

Directly from the definitions, we see that  $\mathbf{f}$ factors through 
$\overline{\Psi}$. We denote by $\overline{\mathbf{f}}$ the induced map from
$\overline{\Psi}$ to $\Sigma$. From Lemma~\ref{lem1}, it follows that 
$\overline{\mathbf{f}}$ is bijective at the level of vertices. 
From the previous paragraph, $\overline{\mathbf{f}}$ is surjective
at the level of labeled oriented edges. Since $\Sigma$ has at most
one oriented edge with a given label between two vertices,
it follows that $\overline{\mathbf{f}}$ is a graph isomorphism, as claimed.
\end{proof}

\begin{proposition}\label{prop4}
The underlying unlabeled oriented graph of $\Sigma$
is simple (i.e. has no double edges). 
\end{proposition}

The following proof is, essentially, due to Microsoft Copilot.

\begin{proof}
Let $i\geq 0$, $N\in\Sigma_i$ and $M\in\Sigma_{i-1}$ be such that
we have an edge labeled by $q$ from $N$ to $M$. Then
$q\in M$ by construction. Consider the set $\Lambda$ of
all elements $z\in M$ such that $z1\in N$. 
Note that $q\in \Lambda$ since the pair $q$ and $q1$ clearly
satisfies Conditions~\eqref{condA} and \eqref{condB}
(note that Condition~\eqref{condC} is automatically satisfied
for all elements in $M$).

Since $z1$ ends with $1$, from Condition~\eqref{condA}
it follows that $\mathrm{supp}(q)\subset \mathrm{supp}(z)$,
for any $z\in\Lambda$. Therefore $q$ is the minimum 
element of $\Lambda$ with respect to inclusion of supports.
This means that we managed to recover the label of our edge 
from $N$ to $M$ purely in terms of $N$ and $M$
as the definition of $\Lambda$ does not use  our label $q$.
The claim of the proposition follows.
\end{proof}

\begin{corollary}\label{cor7}
For $n\geq 0$, the number $R_n$ coincides with the total number of
different oriented paths in $\Sigma$ from
vertices in $\Sigma_n$ to the unique vertex in $\Sigma_{-1}$.
\end{corollary}

\begin{proof}
It is clear from the definitions that the number $R_n$ 
coincides with the total number of different oriented paths in $\Psi$ from
vertices in $\mathcal{R}(n)$ to the unique vertex in $\mathcal{R}(-1)$.
From the quotient construction, it follows that $R_n$ 
coincides with the total number of different oriented paths in 
$\overline{\Psi}$ from vertices in $\mathcal{R}(n)/\sim_n$ to 
the unique vertex in $\mathcal{R}(-1)/\sim_{-1}$.
Since $\overline{\Psi}$ is isomorphic to $\Sigma$ by Theorem~\ref{thm1},
we get the claim.
\end{proof}

\begin{corollary}\label{cor8}
For any vertex in $\Sigma$, all arrows outgoing from this vertex
have the same label.
\end{corollary}

\begin{proof}
From the proof of Proposition~\ref{prop4}, we see that this label
is the following: take our vertex $Q$ and look at all 
$0$-$1$-strings in $Q$ which end with $1$. Among these strings,
there is a unique one with minimal support. Take this string
and delete the rightmost $1$. This is our label.
\end{proof}

Corollary~\ref{cor8} says that the labeling information in $\Sigma$
is, in fact, redundant. One can easily check the assertion
of Corollary~\ref{cor8} on the example provided by Figure~1. 
 We observe that different vertices of $\Sigma$ might have the same
labels for the outgoing edges.

\section{Connection to the Fibonacci sequence}\label{s4}

\subsection{Fibonacci sequence}\label{s4.1}

Recall the Fibonacci sequence $F_n$, where $n\geq 0$,
defined by $F_0=0$, $F_1=1$ and $F_n=F_{n-1}+F_{n-2}$,
for $n\geq 2$, see sequence A000045 in \cite{OEIS}.

We would need the subsequence of the Fibonacci sequence
given by odd indices. Therefore, we define 
$G_n:=F_{2n+1}$, for $n\geq 0$. Note that 
$G_0=1$, $G_1=2$ and $G_n=3G_{n-1}-G_{n-2}$, for $n\geq 2$.

\subsection{Observation}\label{s4.2}

The following observation is due to Microsoft Copilot on
the basis of computations for small $n$: For $n\geq -1$, 
we have $|\Sigma_n|=G_{n+1}$. Our aim in this section is to
prove this for all $n$. The version of Microsoft Copilot
I used was not able to do that.

\subsection{Another recurrence for $G_n$}\label{s4.3}

We will need the following elementary observation:

\begin{proposition}\label{prop21}
We have $G_0=1$, $G_1=2$, $G_2=5$ and
\begin{displaymath}
G_n=2^n-2^{n-2}+G_{n-1}+\sum_{i=1}^{n-2}G_{i}\cdot 2^{n-2-i}, 
\end{displaymath}
for $n>2$.
\end{proposition}

\begin{proof}
We proceed by induction on $n$. The basis $n=3$ 
is easy to check,  indeed, we have
$G_3=13=8-2+5+2$. To prove the induction step, for $n>3$,  we use
$G_n=3G_{n-1}-G_{n-2}$ and the inductive assumption to 
write $G_n$ as
\begin{displaymath}
3\cdot\left(2^{n-1}-2^{n-3}+G_{n-2}+
\sum_{i=1}^{n-3}G_{i}\cdot 2^{n-3-i}\right)- G_{n-2}.
\end{displaymath}
The difference between this expression and our wannabe
formula for $G_n$ equals
\begin{displaymath}
2^{n-1}-2^{n-3}-G_{n-1}+2G_{n-2}+
\sum_{i=1}^{n-3}G_{i}\cdot 2^{n-3-i}-G_{n-2} .
\end{displaymath}
This, however, is equal to $0$ by the inductive assumption.
\end{proof}

\subsection{Vertices of $\Sigma$}\label{s4.4}

For a non-zero $0$-$1$-sequence $z$, we denote by $\mathbf{l}(z)$
the position of the leftmost $1$ in $z$.
For example, we have $\mathbf{l}(00{\color{violet}1}010110)=3$
because of the {\color{violet}violet} letter {\color{violet}$1$} in the sequence. 

\begin{proposition}\label{prop31}
Let $n\geq 0$.
\begin{enumerate}[$($a$)$]
\item\label{prop31.1} The level $\Sigma_n$ contains exactly 
$|\Sigma_{n-1}|$ vertices for which the label of the outgoing arrows
is the zero sequence.
\item\label{prop31.2} If $z$ is a non-zero $0$-$1$ 
sequence of length $n+1$ such that $\mathbf{l}(z)=1$, 
then the level $\Sigma_n$ contains a unique
vertex for which $z$ is the label of the outgoing arrows.
\item\label{prop31.3} If $z$ is a non-zero $0$-$1$ 
sequence of length $n+1$ such that $\mathbf{l}(z)>1$, 
then the level $\Sigma_n$ contains 
exactly $|\Sigma_{\mathbf{l}(z)-3}|$ 
vertices for which $z$ is the label of the outgoing arrows.
\end{enumerate}
\end{proposition}

\begin{proof}
We start with Claim~\eqref{prop31.1}. From the perspective of 
the level $\Sigma_{n-1}$, each vertex at that level contains
the zero sequence $\mathbf{0}_{n+1}$ with $n+1$ letters.
Indeed, for any element $(Y_{n-1},Y_{n-2},\dots, Y_0)\in \mathcal{R}(n-1)$, 
we obviously have that the element 
$(Y_{n-1},Y_{n-2},\dots, Y_0,\mathbf{0}_{n+1})$ belongs to $\mathcal{R}(n)$.
Furthermore, for any $0$-$1$ sequence $u$ at our vertex,
the compatibility Conditions~\eqref{condA} and \eqref{condB}
between $\mathbf{0}_{n+1}$ and both $u0$ and $u1$ are obviously satisfied.
Therefore, starting from $Q\in\Sigma_{n-1}$ and 
going against the arrow labeled by $\mathbf{0}_{n+1}$,
we arrive at the vertex in $\Sigma_n$ which consists exactly of
all $0$-$1$ sequences $u0$ and $u1$, where $u$ is a $0$-$1$-sequence in $Q$.
Consequently, starting from two different vertices $Q$ and $Q'$ in 
$\Sigma_{n-1}$, going against the arrows labeled by 
$\mathbf{0}_{n+1}$ will bring us to two different vertices 
in $\Sigma_{n-1}$.  This implies Claim~\eqref{prop31.1}.

To prove Claim~\eqref{prop31.2}, let $z$ be a non-zero $0$-$1$ 
sequence of length $n+1$ whose leftmost letter is $1$. Write 
$\mathrm{supp}(z)=\{i_1=1,i_2,\dots,i_p\}$ such that 
$i_1<i_2<\dots <i_p$. We claim that the starting vertex $Q$
of $z$ consists of all $0$-$1$ sequences $u$ whose support
is an initial segment of $\{i_1,i_2,\dots,i_p,n+2\}$.
If we can prove that, we get Claim~\eqref{prop31.2} as we
manifestly show that $Q$ is uniquely determined by $z$.

Our first step is to show that the support of any $u\in Q$
is an initial segment of $\{i_1,i_2,\dots,i_p,n+2\}$. This 
follows from the compatibility Conditions~\eqref{condA}
and \eqref{condB} applied to $z$ and $u$. Indeed, if $u$
ends with $1$, then Condition~\eqref{condA} says exactly that 
$\mathrm{supp}(z)\subset\mathrm{supp}(u)$. If 
$j\in \mathrm{supp}(u)\setminus (\mathrm{supp}(z)\cup\{n+2\})$, 
then we can apply Condition~\eqref{condB} to position $1$ in $z$, 
where we have $1$, and position $j\in u$, where we also have $1$.
Since $j\neq n+2$ by our assumptions, we must have $1$ at  position
$j$ in $z$, a contradiction. Therefore, in this case, 
$\mathrm{supp}(u)$ is the union of $\mathrm{supp}(z)$
and $\{n+2\}$. 

If $u$ ends with $0$, take any $j\in \mathrm{supp}(u)$.
Since $1\in \mathrm{supp}(z)$,
from Condition~\eqref{condB} we have $j\in \mathrm{supp}(z)$,
so $\mathrm{supp}(u)\subset\mathrm{supp}(z)$. Let
$j'\in \mathrm{supp}(z)$ be such that $j'<j$. Then we 
apply Condition~\eqref{condB} again and get 
$j'\in \mathrm{supp}(u)$. This completes the proof of the claim
that the support of any $u\in Q$
is an initial segment of $\{i_1,i_2,\dots,i_p,n+2\}$.

Now let us show that every $0$-$1$ string of length $n+2$
whose support is an initial segment of $\{i_1,i_2,\dots,i_p,n+2\}$
belongs to $Q$. So, let $u$ be a $0$-$1$ string of length $n+2$
whose support is an initial segment of $\{i_1,i_2,\dots,i_p,n+2\}$. 
Such $u$ obviously satisfies  the compatibility 
Conditions~\eqref{condA} and \eqref{condB}. Deleting the last 
letter from $u$, we get a $0$-$1$ string, which we call $w$, 
of length $n+1$, whose support is an initial segment of $\mathrm{supp}(z)$.
We need to prove that $w$ belongs to the target vertex $Q'$ of our arrow
labeled by $z$. Note that, by definition, $z$ belongs to $Q'$.
Therefore, what we really need is the following lemma
(this lemma was suggested by Microsoft Copilot).

\begin{lemma}\label{lemDW}
Let $R$ be any vertex in $\Sigma$ and $x\in R$. 
Let $y$ be a $0$-$1$-string of the same length as $x$
such that $\mathrm{supp}(y)$ is an initial segment of 
$\mathrm{supp}(x)$. Then $y\in R$.
\end{lemma}

\begin{proof}
We proceed by induction on the level, with the basis of induction being
trivial. To prove the induction step, take any vertex $R'$ of $\Sigma$
and any $x'\in R'$. Let $y'$ be a $0$-$1$-string of the same length as $x'$
satisfying that $\mathrm{supp}(y')$ is an initial segment of 
$\mathrm{supp}(x')$. We assume that $y'$ is different from $x'$.

By our recursive definition of $\Sigma$, the vertex $R'$
was constructed from some vertex $R''$ of a smaller level and
some element $q\in R''$. Our element $x'$ has the form
$q'0$ or $q'1$, for some $q'\in R''$ such that 
Conditions~\eqref{condA} and \eqref{condB} are satisfied.

Deleting the last letter from $y'$, we get an element which 
we denote by $y''$. It is a $0$-$1$-string of the same length as $q'$
satisfying that $\mathrm{supp}(y'')$ is an initial segment of 
$\mathrm{supp}(q')$. By induction, $y''\in R''$.

If the last letter of $y'$ were $1$, the initial segment condition
would force $y'=x'$ and we excluded this case. Therefore
the last letter of $y'$ is $0$. Adding this $0$ to $y''$ 
on the right, on the one hand, outputs $y'$, on the other hand,
satisfies Conditions~\eqref{condA} and \eqref{condB} since 
$\mathrm{supp}(y'')$ is an initial segment of 
$\mathrm{supp}(q')$ and $q'$
satisfies Conditions~\eqref{condA} and \eqref{condB}.
Therefore $y'\in R'$.
The proof of our induction step is now complete.
\end{proof}

Application of Lemma~\ref{lemDW} finishes the proof of Claim~\eqref{prop31.2}.

Claim~\eqref{prop31.3} reduces to a combination of 
Claim~\eqref{prop31.1} and Claim~\eqref{prop31.2}.
Indeed, let $z$ be a non-zero $0$-$1$ sequence.
We assume that $z$ is the label of an arrow from some 
vertex $Q$ to some vertex $R$ in $\Sigma$. 
Let us look at the following pictorial representation of the 
elements in $\mathcal{R}(n)$ which can be obtained as paths starting
from the vertex $Q$:
\begin{displaymath}
\xymatrix@R=0.5mm@C=0.5mm{
{\color{teal}\ast}&&&&&&&\\
{\color{teal}\ast}&{\color{teal}\ast}&&&&&&\\
{\color{teal}\ast}&{\color{teal}\ast}&{\color{teal}\ast}&&&&&\\
0&0&0&0&&&&\\
0&0&0&0&{\color{violet}\diamond}&&&\\
0&0&0&0&{\color{violet}\diamond}&{\color{violet}\diamond}&&\\
0&0&0&0&{\color{violet}\diamond}&{\color{violet}\diamond}&{\color{violet}\diamond}\\
0&0&0&0&{\color{violet}1}&{\color{violet}\diamond}&
{\color{violet}\diamond}&{\color{violet}\diamond}
}
\end{displaymath}
Here we see our $z$ as the bottom row with the leftmost $1$
at position $\mathbf{l}(z)$. Because of this $1$, Condition~$(\star)$
forces the South-West rectangle filled with zeros. This splits the whole
picture into two parts: the {\color{teal}teal} North-West triangle and
the {\color{violet}violet} South-East triangle. We see that there is 
the zero sequence directly 
under the North-West triangle. This is exactly the situation
treated in Claim~\eqref{prop31.1}. It produces, for the North-West triangle, 
the  $|\Sigma_{\mathbf{l}(z)-3}|$ factor by directly applying
Claim~\eqref{prop31.1}.

We also see the $1$ in the South-West corner of the 
South-East triangle. This is exactly the situation
treated in Claim~\eqref{prop31.2}. It produces, for the
South-East triangle,  the 
factor $1$ by directly applying Claim~\eqref{prop31.2}.
Claim~\eqref{prop31.3} now follows by applying the
multiplication rule.
\end{proof}

\subsection{Proof of Copilot's conjecture}\label{s4.9}

\begin{theorem}\label{thm3}
For $n\geq -1$,  we have $|\Sigma_n|=G_{n+1}$.  
\end{theorem}

\begin{proof}
For small values of $n$, this is easy to check by hand.
Due to Corollary~\ref{cor8}, Propositions~\ref{prop31} implies that
$|\Sigma_n|$ satisfies the recurrence
\begin{displaymath}
|\Sigma_{n}|=
|\Sigma_{n-1}|+2^n+\sum_{i=2}^{n+1}|\Sigma_{i-3}|\cdot 2^{n+1-i}. 
\end{displaymath}
Therefore, the claim of our theorem follows by comparing the above recurrence with the
appropriately reindexed recurrence in  Proposition~\ref{prop21}.
\end{proof}

\section{Some growth estimates}\label{s7}

\subsection{Observations from our computations}\label{s7.1}

Based on our computations above (and similar computations
using SageMath for bigger values of $n$), here are the approximate
values for the ratios $\frac{R_{n+1}}{R_{n}}$, for 
$n=0,1,2,\dots,10,20,30,\dots,100$:

\resizebox{\textwidth}{!}{
$
\begin{array}{c||c|c|c|c|c|c|c|c|c|c|c}
n & 0 & 1 & 2 & 3 & 4 & 5 & 6 & 7 & 8 & 9 & 10 \\  \hline\hline
\resizebox{5mm}{!}{$\frac{R_{n+1}}{R_n}$} & 
3.50 & 4.86 & 5.85 & 6.57 & 7.11 & 7.52 
& 7.85 & 8.12 & 8.34 & 8.53 & 8.69 
\end{array}
$
}

\resizebox{\textwidth}{!}{
$
\begin{array}{c||c|c|c|c|c|c|c|c|c}
n & 20 & 30 & 40 & 50 & 60 & 70 & 80 & 90 & 100\\
\hline\hline
\resizebox{5mm}{!}{$\frac{R_{n+1}}{R_n}$}  & 9.53 & 9.87 & 10.05 & 10.17 
& 10.25 & 10.30 & 10.35 & 10.38 & 10.41
\end{array}
$
}

This sequence of ratios does look converging with the limit 
probably somewhere around $10.5$. It looks as an interesting
question to find some reasonable and non-trivial lower and upper
bounds for $R_n$. It seems that both problems are not really easy.

For comparison, the sequence of the approximate values of $\sqrt[n]{R_n}$
is as follows:

\resizebox{\textwidth}{!}{
$
\begin{array}{c||c|c|c|c|c|c|c|c|c|c}
n & 1 & 2 & 3 & 4 & 5 & 6 & 7 & 8 & 9 & 10 \\
\hline\hline
\resizebox{5mm}{!}{$\sqrt[n]{R_n}$} & 7.00 & 5.83 
& 5.84 & 6.01 
& 6.22 & 6.42 & 6.61 & 6.78 & 6.94 & 7.08 
\end{array}
$
}

\resizebox{\textwidth}{!}{
$
\begin{array}{c||c|c|c|c|c|c|c|c|c}
n &  20 & 30 & 40 & 50 & 60 & 70 & 80 & 90 & 100\\
\hline\hline
\resizebox{5mm}{!}{$\sqrt[n]{R_n}$} & 8.04 & 8.56 & 8.89 
& 9.12 & 9.30 & 9.43 & 9.54 & 9.63 & 9.70
\end{array}
$
}

This also looks converging with the limit probably somewhere 
between $10$ and $11$.  Of course, it is reasonable to 
expect that both sequences converge to the same limit.

\subsection{$R_n$ vs $P_n$}\label{s7.2}
 
Let us write $\mathcal{R}(n)$ as the disjoint union
\begin{displaymath}
\mathcal{R}(n)= \mathcal{P}(n)\cup \mathcal{S}(n),
\end{displaymath}
where
\begin{displaymath}
\mathcal{S}(n):=\{X\in \mathcal{R}(n)\,:\, (0,0)\not\in X\}. 
\end{displaymath}
Then $|\mathcal{R}(n)|=|\mathcal{P}(n)|+|\mathcal{S}(n)|$
by the Rule of Sum. Mapping $X\in \mathcal{S}(n)$ to
$X\cup\{(0,0)\}$ defines an injective map from 
$\mathcal{S}(n)$ to $\mathcal{P}(n)$. Therefore we have
\begin{equation}\label{eq:rnpm}
|\mathcal{S}(n)|\leq |\mathcal{P}(n)| 
\end{equation}
and thus
\begin{displaymath}
|\mathcal{P}(n)|\leq R_n\leq 2 |\mathcal{P}(n)|.
\end{displaymath}
Consequently, $R_n$ and $P_n$ have the same rate of growth.

\subsection{Lower bound}\label{s7.3}

In this section we provide a lower bound for the sequence $R_n$,

\begin{remark}
{\em Microsoft Copilot asserts that 
$R_n\geq \frac{1}{111}\cdot 10^n$, using methods similar
to the ones described below, but based on a computation
for $L=160$, the method below uses $L=47$. I was not able to 
verify this $10^n$ lower bound using SageMath
as the corresponding computation simply took too long.
Therefore below I present a weaker $9^n$ lower bound which
I was able to verify.
}
\end{remark}

Our goal is to prove the following:

\begin{theorem}\label{thm:lbound}
For every integer $n\ge 0$, one has
\begin{equation}\label{eq:lower-9}
 R_n\ge \frac1{12}\,9^n.
\end{equation}
\end{theorem}

The proof is, essentially, due to Microsoft Copilot.
The computer assisted part is checked independently using
SageMath.

For $t\ge 0$, define $\displaystyle W_t:=\sum_{p\ge 1} (p+1)b(t,p)$.
Since $b(t,p)=0$, for $p>t+1$, the sum is finite.

\begin{lemma}\label{lem:P-exact}
For every $n\ge 1$, we have
\begin{equation}\label{eq:P-exact}
 P_n = R_{n-1} + \sum_{t=0}^{n-1} W_t\,R_{n-t-2}.
\end{equation}
\end{lemma}

\begin{proof}
Sum \eqref{recurrence:b} over all $k\ge 1$.
The first term contributes
\[
 \sum_{k\ge 1} a(n-1,k-1)=\sum_{r\ge 0} a(n-1,r)=R_{n-1}.
\]
Now fix $m,r,p$ in the double sum.  The integer $k$ contributes precisely when
\[
 r\le k-1\le r+p,
\]
that is, for exactly the $p+1$ values
\[
 k=r+1,r+2,\dots,r+p+1.
\]
Therefore each fixed pair $(r,p)$ contributes a multiplicity factor $p+1$ after summation over $k$.  Summing first over $r$ then gives $R_{m-2}$, and writing $t=n-m$ yields
\[
 P_n = R_{n-1} + \sum_{m=1}^n R_{m-2}\sum_{p\ge 1}(p+1)b(n-m,p)
 =R_{n-1}+\sum_{t=0}^{n-1} W_t\,R_{n-t-2},
\]
as claimed.
\end{proof}

\begin{corollary}\label{cor:R-lower-rec}
For every $n\ge 1$, one has
\begin{equation}\label{eq:R-lower-full}
 R_n \ge 2R_{n-1} + \sum_{t=0}^{n-1} W_t\,R_{n-t-2}.
\end{equation}
In particular, for every fixed integer $L\ge 0$ and every $n\ge L+2$, we have
\begin{equation}\label{eq:R-lower-trunc}
 R_n \ge 2R_{n-1} + \sum_{t=0}^L W_t\,R_{n-t-2}.
\end{equation}
\end{corollary}

\begin{proof}
Summing \eqref{recurrence:a2} 
over all $k\ge 1$ and adding the special term $a(n,0)=R_{n-1}$ gives
\[
 R_n = R_{n-1} + P_n + \sum_{j=0}^{n-1} P_j R_{n-j-2}.
\]
All terms are nonnegative, hence $R_n \ge R_{n-1} + P_n$.
Replace $P_n$ using \eqref{eq:P-exact} to obtain \eqref{eq:R-lower-full}.  
The truncated inequality \eqref{eq:R-lower-trunc} follows by 
discarding the nonnegative terms with $t>L$.
\end{proof}

Let us now set $L=47$. Using the defining recurrence for
$b(n,k)$, one computes the exact integers $W_t$ for $0\le t\le 47$:
\begin{gather*}
W_{0}=2,\; W_{1}=10,\; W_{2}=61,\; W_{3}=417, \;
W_{4}=3069,\; W_{5}=23792,\; W_{6}=191729,\\ W_{7}=1592215, \;
W_{8}=13544540,\; W_{9}=117514054,\; W_{10}=1036498196,\\ 
W_{11}=9270803468, \; W_{12}=83923822932,\; W_{13}=767698692000,\\ 
W_{14}=7087279559385,\; W_{15}=65962045933923, \;
W_{16}=618375778857162,\\ 
W_{17}=5834883859214502,\; W_{18}=55380727815966261,\; 
W_{19}=528440108649254193, 
\end{gather*}
\begin{gather*}
W_{20}=5066863647277382769,\; 
W_{21}=48798993733336365360,\\ W_{22}=471904895014440160956,\; 
W_{23}=4580704144611974549252, \\
W_{24}=44619216594372104854304,\; 
W_{25}=436028318197971420102568,\\ 
W_{26}=4273788952292896870660328,\; 
W_{27}=42007759647196604443899208, \\
W_{28}=413985448734460309113392408,\; 
W_{29}=4089858352463024931919779008,\\ 
W_{30}=40498057056226978339012990953,\; 
W_{31}=401886777330433575319948408443, 
\end{gather*}
\begin{gather*}
W_{32}=3996343833415466695413701072038,\\ 
W_{33}=39816410241921345855018229262518,\\ 
W_{34}=397425795703252401024720411455543,\\ 
W_{35}=3973771342288571047788883369035379, \\
W_{36}=39798250184133255692570181051276767,\\
W_{37}=399212091112896726572027912994238096,\\ 
W_{38}=4010411552498593848908082035772086089, 
\end{gather*}
\begin{gather*}
W_{39}=40344938685751664890290577708518542799, \\
W_{40}=406420400488580295940123687670896750452,\\ 
W_{41}=4099413005919968859471552690203422611126,\\ 
W_{42}=41400232013530068435300464857128490968696,\\ 
W_{43}=418596518041416171957321505623350092945200, \\
W_{44}=4237195258436912452721733225720491584886352,\\ 
W_{45}=42936933575880059491791692606286077976763008,\\ 
W_{46}=435546123810287228552654386737526592379323484,\\ 
W_{47}=4422521027852755507048330045052028456829333620.
\end{gather*}

Associated with the truncated linear recurrence 
\eqref{eq:R-lower-trunc} is the polynomial
\begin{equation}\label{eq:F47-def}
 F_{47}(x):=x^{49}-2x^{48}-\sum_{t=0}^{47} W_t x^{47-t}.
\end{equation}
Evaluating at $x=9$ gives the exact integer
\begin{equation}\label{eq:F47-negative}
 F_{47}(9)=-2179856339179759119326891523960197013600439716<0.
\end{equation}
Consequently,
\begin{equation}\label{eq:9-subsolution}
 9^{49}\le 2\cdot 9^{48}+\sum_{t=0}^{47} W_t 9^{47-t}.
\end{equation}
Multiplying \eqref{eq:9-subsolution} by $9^{n-49}$ gives, for every $n\ge 49$,
\begin{equation}\label{eq:9-subsolution-general}
 9^n \le 2\cdot 9^{n-1}+\sum_{t=0}^{47} W_t 9^{n-t-2}.
\end{equation}
Thus the sequence $9^n$ is a subsolution for the linear recurrence \eqref{eq:R-lower-trunc}.

Exact computation gives
\[
 12R_n-9^n\ge 0 \qquad (0\le n\le 48).
\]
Moreover, the minimum value in this range is attained at $n=0$ and equals
\begin{equation}\label{eq:min-initial}
 \min_{0\le n\le 48} \bigl(12R_n-9^n\bigr) = 23 > 0.
\end{equation}
Hence
\begin{equation}\label{eq:initial-bound}
 R_n\ge \frac1{12}9^n \qquad (0\le n\le 48).
\end{equation}

We are now ready to prove Theorem~\ref{thm:lbound}.

\begin{proof}[Proof of Theorem~\ref{thm:lbound}.]
We prove \eqref{eq:lower-9} by induction on $n$.  The initial range $0\le n\le 48$ is exactly \eqref{eq:initial-bound}.
Now let $n\ge 49$ and assume inductively that
\[
 R_m\ge \frac1{12}9^m \qquad (0\le m<n).
\]
Then, by the truncated lower recurrence \eqref{eq:R-lower-trunc} with $L=47$,
\[
 R_n \ge 2R_{n-1}+\sum_{t=0}^{47} W_t R_{n-t-2}.
\]
Applying the induction hypothesis to each term on the right gives
\[
 R_n \ge \frac1{12}\left(2\cdot 9^{n-1}+\sum_{t=0}^{47} W_t 9^{n-t-2}\right).
\]
Finally, we use \eqref{eq:9-subsolution-general} to conclude that
$R_n\ge \frac1{12}9^n$.
This completes the induction and proves \eqref{eq:lower-9}.  
\end{proof}

\begin{remark}
{\em The above  proof of Theorem~\ref{thm:lbound} 
is computer-assisted in only two places:
\begin{itemize}
\item the exact integers $W_0,\dots,W_{47}$ are computed from the recurrences;
\item the exact inequalities \eqref{eq:F47-negative} and \eqref{eq:min-initial} are verified by exact integer arithmetic.
\end{itemize}
Both are checked via an independent SageMath computation.
}
\end{remark}

\subsection{Upper bound}\label{s7.4}
The analysis in this subsection is, essentially, due to Microsoft Copilot.
We refer to \cite{FS} as a general reference for 
details on analysis of generating functions.

Fix a real number $c>1$.  Define
\[
 X_n:=\sum_{k\ge 0} c^k a(n,k),
 \qquad
 Y_n:=\sum_{k\ge 1} c^k b(n,k)
 \qquad (n\ge 0),
\]
and set also
\[
 X_{-1}:=\sum_{k\ge 0} c^k a(-1,k)=1.
\]
Since all coefficients are nonnegative, we immediately have
\begin{equation}\label{eq:RXcompare}
 A(n)\le X_n,
 \qquad
 B(n)\le Y_n
 \qquad (n\ge 0).
\end{equation}
It will be convenient to shift the $X$-sequence by one index and write
\[
 U_n:=X_{n-1} \qquad (n\ge 0),
 \qquad
 V_n:=Y_n \qquad (n\ge 0).
\]
Thus $U_0=X_{-1}=1$ and
\[
 U_1=X_0=a(0,0)+ca(0,1)=1+c,
 \qquad
 V_0=Y_0=cb(0,1)=c.
\]

\begin{lemma}\label{lem:weighted-ineq}
For every $n\ge 1$, we have
\begin{equation}\label{eq:U-ineq}
 U_{n+1}\le U_n+V_n+\sum_{j=0}^{n-1} V_j\,U_{n-1-j},
\end{equation}
and
\begin{equation}\label{eq:V-ineq}
 V_n\le cU_n+\frac{c^2}{c-1}\sum_{j=0}^{n-1} U_j\,V_{n-1-j}.
\end{equation}
\end{lemma}

\begin{proof}
We first prove \eqref{eq:U-ineq}. We have
\[
 a(n,0)=A(n-1)\le X_{n-1}=U_n.
\]
For $k\ge 1$, multiply \eqref{recurrence:a2} by $c^k$ and sum over $k\ge 1$:
\[
 \sum_{k\ge 1} c^k a(n,k)
 =\sum_{k\ge 1} c^k b(n,k)
 +\sum_{j=0}^{n-1}\Bigl(\sum_{k\ge 1} c^k b(j,k)\Bigr)
 \Bigl(\sum_{i=0}^{n-j-1} a(n-j-2,i)\Bigr).
\]
The inner sum over $i$ is bounded by $X_{n-j-2}=U_{n-1-j}$, hence
\[
 \sum_{k\ge 1} c^k a(n,k)
 \le V_n+\sum_{j=0}^{n-1} V_j U_{n-1-j}.
\]
Adding the contribution of $k=0$ gives
\[
 X_n\le X_{n-1}+V_n+\sum_{j=0}^{n-1} V_j U_{n-1-j},
\]
that is,
\[
 U_{n+1}\le U_n+V_n+\sum_{j=0}^{n-1} V_j U_{n-1-j}.
\]
This proves \eqref{eq:U-ineq}.

We now prove \eqref{eq:V-ineq}.  Start from \eqref{recurrence:b}, 
multiply by $c^k$, and sum over $k\ge 1$.
The first term is
\[
 \sum_{k\ge 1} c^k a(n-1,k-1)
 =c\sum_{r\ge 0} c^r a(n-1,r)
 =cX_{n-1}=cU_n.
\]
For the second term, fix $m,r,p$.  The constraints in 
\eqref{recurrence:b} imply that the admissible values of $k$ satisfy
\[
 r+1\le k\le r+p+1.
\]
Therefore
\[
 \sum_{k=r+1}^{r+p+1} c^k
 =c^{r+1}\frac{c^{p+1}-1}{c-1}
 \le \frac{c^{r+1}c^{p+1}}{c-1}
 =\frac{c^2}{c-1}c^r c^p.
\]
Using this estimate and summing over all $m,r,p$ gives
\[
 Y_n
 \le cX_{n-1}
 +\frac{c^2}{c-1}
 \sum_{m=1}^n
 \Bigl(\sum_{r\ge 0} c^r a(m-2,r)\Bigr)
 \Bigl(\sum_{p\ge 1} c^p b(n-m,p)\Bigr).
\]
Hence
\[
 Y_n\le cX_{n-1}+\frac{c^2}{c-1}\sum_{m=1}^n X_{m-2}Y_{n-m}.
\]
Replacing $X_{m-2}$ by $U_{m-1}$ and then relabelling $j=m-1$ yields
\[
 V_n\le cU_n+\frac{c^2}{c-1}\sum_{j=0}^{n-1} U_j V_{n-1-j},
\]
as claimed.
\end{proof}

Define sequences $\widehat U_n$, $\widehat V_n$ by
\[
 \widehat U_0=1,
 \qquad
 \widehat U_1=1+c,
 \qquad
 \widehat V_0=c,
\]
and for every $n\ge 1$,
\begin{equation}\label{eq:U-majorant}
 \widehat U_{n+1}=\widehat U_n+\widehat V_n+\sum_{j=0}^{n-1} \widehat V_j\,\widehat U_{n-1-j},
\end{equation}
\begin{equation}\label{eq:V-majorant}
 \widehat V_n=c\widehat U_n+\frac{c^2}{c-1}\sum_{j=0}^{n-1} \widehat U_j\,\widehat V_{n-1-j}.
\end{equation}
Because the right-hand sides are monotone in all variables and the initial inequalities are equalities, an induction on $n$ yields
\begin{equation}\label{eq:majorization}
 U_n\le \widehat U_n,
 \qquad
 V_n\le \widehat V_n
 \qquad (n\ge 0).
\end{equation}

Now introduce the ordinary generating functions
\[
 \widehat U(z):=\sum_{n\ge 0} \widehat U_n z^n,
 \qquad
 \widehat V(z):=\sum_{n\ge 0} \widehat V_n z^n.
\]

\begin{lemma}\label{lem:GF-system}
The generating functions $\widehat U(z)$ and $\widehat V(z)$ satisfy
\begin{equation}\label{eq:GF1}
 \widehat U(z)=1+z\widehat U(z)+z\widehat V(z)+z^2\widehat U(z)\widehat V(z),
\end{equation}
\begin{equation}\label{eq:GF2}
 \widehat V(z)=c\widehat U(z)+\frac{c^2}{c-1}z\widehat U(z)\widehat V(z).
\end{equation}
\end{lemma}

\begin{proof}
Multiply \eqref{eq:U-majorant} by $z^n$ and sum over $n\ge 1$.  The left-hand side is
\[
 \sum_{n\ge 1} \widehat U_{n+1} z^n
 =\frac{\widehat U(z)-\widehat U_0-\widehat U_1 z}{z}.
\]
The first two terms on the right-hand side give
\[
 \sum_{n\ge 1} \widehat U_n z^n=\widehat U(z)-\widehat U_0,
 \qquad
 \sum_{n\ge 1} \widehat V_n z^n=\widehat V(z)-\widehat V_0.
\]
The convolution term becomes
\[
 \sum_{n\ge 1}\sum_{j=0}^{n-1}\widehat V_j\widehat U_{n-1-j}z^n
 =z\widehat U(z)\widehat V(z).
\]
Hence
\[
 \frac{\widehat U(z)-\widehat U_0-\widehat U_1 z}{z}
 =\widehat U(z)-\widehat U_0+\widehat V(z)-\widehat V_0+z\widehat U(z)\widehat V(z).
\]
Substituting $\widehat U_0=1$, $\widehat U_1=1+c$, $\widehat V_0=c$ and simplifying gives \eqref{eq:GF1}.

Similarly, multiply \eqref{eq:V-majorant} by $z^n$ and sum over $n\ge 1$.  Then
\[
 \sum_{n\ge 1}\widehat V_n z^n=\widehat V(z)-\widehat V_0,
\qquad
 c\sum_{n\ge 1}\widehat U_n z^n=c\bigl(\widehat U(z)-\widehat U_0\bigr),
\]
and
\[
 \sum_{n\ge 1}\sum_{j=0}^{n-1}\widehat U_j\widehat V_{n-1-j}z^n
 =z\widehat U(z)\widehat V(z).
\]
Thus
\[
 \widehat V(z)-\widehat V_0
 =c\bigl(\widehat U(z)-\widehat U_0\bigr)
 +\frac{c^2}{c-1}z\widehat U(z)\widehat V(z).
\]
Using $\widehat U_0=1$ and $\widehat V_0=c$ yields \eqref{eq:GF2}.
\end{proof}

From now on, we set
\[
 c=\frac32,
 \qquad \text{ then }\qquad
 \frac{c^2}{c-1}=\frac92.
\]
In this case, \eqref{eq:GF2} becomes
\[
 \widehat V(z)=\frac32\,\widehat U(z)+\frac92 z\widehat U(z)\widehat V(z).
\]
Hence
\begin{equation}\label{eq:V-in-terms-of-U}
 \widehat V(z)=\frac{\frac32\,\widehat U(z)}{1-\frac92 z\widehat U(z)}.
\end{equation}
Substituting \eqref{eq:V-in-terms-of-U} into \eqref{eq:GF1} and simplifying gives a quadratic equation for $\widehat U(z)$:
\begin{equation}\label{eq:quadratic-U}
 6z^2\widehat U(z)^2-9z\widehat U(z)^2+4z\widehat U(z)+2\widehat U(z)-2=0.
\end{equation}
Equivalently,
\[
 \bigl(9z-6z^2\bigr)\widehat U(z)^2-(2+4z)\widehat U(z)+2=0.
\]
Solving for the branch regular at $z=0$ and satisfying $\widehat U(0)=1$ yields
\begin{equation}\label{eq:explicit-U}
 \widehat U(z)=\frac{1+2z-\sqrt{1-14z+16z^2}}{3z(3-2z)}.
\end{equation}
The dominant singularity therefore comes from the vanishing of
\[
 1-14z+16z^2=0,
\]
whose smaller positive root is
\begin{equation}\label{eq:rho-1}
 \rho=\frac{7-\sqrt{33}}{16}.
\end{equation}
Hence
\[
 \rho^{-1}=7+\sqrt{33}\approx 12.74456.
\]
This already implies that the majorant sequence $\widehat U_n$ grows at most exponentially with any base strictly larger than $7+\sqrt{33}$.
Altogether, we have proved the following:

\begin{theorem}\label{thm:ub}
Let $\varepsilon>0$. Then, for all $n\gg 0$, we have
$R_n< (7+\sqrt{33}+\varepsilon)^n$.
\end{theorem}

\subsection{Consequence}\label{s7.5}

Combining Theorem~\ref{thm:lbound} with 
Theorem~\ref{thm:ub}, we have:

\begin{corollary}\label{cor:growth}
For $n\gg 0$, we have 
$9^n < R_n< 12.75^n$. 
\end{corollary}

\begin{proof}
The only non-trivial thing here to note is that 
from the proof of Theorem~\ref{thm:lbound} we see
that the maximal root of $F_{47}(x)$ 
is strictly bigger than $9$, however, we
do not know exactly what it is. Therefore  our sequence
grows exponentially with some base strictly greater
than $9$, so that we can remove the factor $\frac{1}{12}$
and make the inequality on the left hand side strict for $n\gg 0$.
\end{proof}

\section{Further connections}\label{s9}

\subsection{Connection to naturally labeled partial orders}\label{s9.1}

The sequence A006455 in \cite{OEIS} enumerates naturally labeled partial orders
on $\underline{n}$, for $n\in\mathbb{Z}_{>0}$. A partial order on 
$\underline{n}$ is {\em naturally labeled} provided that it is a subset of
the usual (strict) linear order on $\underline{n}$. This usual linear order
consists of all pairs $(x,y)\in \underline{n}^2$ such that $x<y$.

Sending $(x,y)\in \underline{n}^2$ such that $x<y$ to 
$(x-1,n-1-y)\in\mathcal{Q}(n-2)$ defines a bijection between the 
usual linear order on $\underline{n}$ and $\mathcal{Q}(n-2)$.
Given a naturally labeled partial order on $\underline{n}$,
its transitivity translates to the following condition for 
the corresponding subset $Y$ of $\mathcal{Q}(n-2)$: if
$(a,b)$ and $(c,d)$ are in $Y$ and $(\max(a,c),\max(b,d))=n-1$,
then $(\min(a,c),\min(b,d))$ is in $Y$.
All elements in $\mathcal{R}(n-2)$ satisfy this condition and
hence can be interpreted as naturally labeled partial orders
on $\underline{n}$.

We note that the sequence A006455 starts off as follows:
\begin{displaymath}
1,\, 2,\, 7,\, 40,\, 357,\, 4824,\, 96428 
\end{displaymath}
and hence quickly diverges from $R_{n-2}$.
The growth of A006455 is explicitly described  in \cite{BPS},
it is superexponential with the 
leading term $2^{n^2/4}$. 
%
%
%

\subsection{Connection to Catalan numbers}\label{s9.3}

From the computed values of $a$, one could observe that
$a(n,n+1)$ seems to be a Catalan number. This connection is explained
in this subsection with all details.

For any $n\in\mathbb{Z}_{\geq 0}$, 
let $\mathcal{R}'(n)$ denote the set of all 
$Y=(Y_n,Y_{n-1},\dots,Y_0)\in\mathcal{R}(n)$ satisfying
$Y_0=\{0,1,2,\dots,n\}$. It turns out that 
the sets $\mathcal{R}'(n)$, for $n\in\mathbb{Z}_{\geq 0}$, 
are enumerated by Catalan numbers (sequence 
A000108 in \cite{OEIS}). For example, here 
are the five elements of $\mathcal{R}'(2)$:
\begin{displaymath}
\begin{array}{ccc}0&&\\0&0&\\1&1&1\end{array}\qquad 
\begin{array}{ccc}0&&\\0&1&\\1&1&1\end{array}\qquad 
\begin{array}{ccc}0&&\\1&1&\\1&1&1\end{array}\qquad 
\begin{array}{ccc}1&&\\0&0&\\1&1&1\end{array}\qquad 
\begin{array}{ccc}1&&\\1&1&\\1&1&1\end{array}.
\end{displaymath}
A significant amount of details in the arguments below
is due to Microsoft Copilot, however, the final version
did require some correction and insertion of more details.

\begin{proposition}\label{prop:catalan}
We have $R'_n:=|\mathcal{R}'(n)|=\frac{1}{n+2}\binom{2n+2}{n+1}$.  
\end{proposition}

\begin{proof}
It is easy to check that $R'_0=1$, $R'_1=2$ and $R'_2=5$.
To prove the claim, we will show that the sequence
$R'_n$ satisfies the usual recurrence for
Catalan numbers.

Fix $n\ge 0$ and let $Y=(Y_n,Y_{n-1},\dots,Y_0)\in \mathcal{R}'(n)$.
For each $j\in\{1,2,\dots,n\}$, consider the $j$-th horizontal row
\[
H_j:=\{x\in\{0,1,\dots,n-j\} : (x,j)\in Y\}.
\]
Our first observation
is that, for any element $j\in\{1,2,\dots,n\}$, 
the set $H_j$ is either empty or a suffix of $\{0,1,\dots,n-j\}$.
Equivalently, there exists a number
\[
h_j\in\{0,1,\dots,n-j\}\cup\{\infty\}
\]
such that
\[
H_j=\{x\in\{0,1,\dots,n-j\}: x\ge h_j\},
\]
with the convention that $H_j=\varnothing$ when $h_j=\infty$.
Indeed, let $(x,j)\in Y$ and $x'\ge x$ with $x'+j\le n$. 
Since $(x',0)\in Y$, Condition~$(\star)$ forces $(x',j)\in Y$.
This means that our set $Y$ is completely determined by its 
{\em threshold sequence}
$(h_n,\dots,h_2,h_1)$, where all $h_j$ are defined above.
Now we make the following observation.

\begin{lemma}\label{lem:comparison}
Let $Y=(Y_n,\dots,Y_0)\subseteq \mathcal{Q}(n)$ be such that
$Y_0=\{0,\dots,n\}$, and let $(h_n,\dots,h_1)$ be its threshold sequence.
As above, we assume that each $H_j$ is either empty or a suffix of $\{0,1,\dots,n-j\}$.
Then $Y\in \mathcal{R}(n)$ if and only if, for all pairs $1\le i<j\le n$, one has
\begin{equation}\label{eq:comp}
 h_i\le n+1-j \quad\Longrightarrow\quad h_i\le h_j.
\end{equation}
Equivalently, $h_i>h_j$ implies $h_i+j\ge n+2$.
\end{lemma}

\begin{proof}
We first assume $Y\in \mathcal{R}(n)$ and prove \eqref{eq:comp}. 
Suppose $i<j$ and $h_i\le n+1-j$. If $h_j=\infty$, then automatically $h_i\le h_j$, so there is nothing to prove. Assume therefore that $h_j<\infty$. The points $(h_i,i)$ and $(h_j,j)$ both lie in $Y$, and
\[
\max(h_i,h_j)+\max(i,j)\le h_i+j\le n+1
\]
whenever $h_i>h_j$. Applying Condition~$(\star)$ to these two points would then force
\[
(\min(h_i,h_j),\min(i,j))=(h_j,i)\in Y,
\]
contradicting the minimality of $h_i$. Hence $h_i\le h_j$. This proves \eqref{eq:comp}.

Conversely, assume the thresholds satisfy \eqref{eq:comp}. 
We need to show that $Y\in \mathcal{R}(n)$. 
Let $(a,i),(c,j)\in Y$ with $i\le j$ and assume
\begin{equation}\label{eq:assumtion}
\max(a,c)+j\le n+1.
\end{equation}
We need to show that $(\min(a,c),i)\in Y$ and, additionally,  
that $(\max(a,c),j)\in Y$ whenever the latter lies in $\mathcal{Q}(n)$.

If $a\le c$, then $(\min(a,c),i)=(a,i)$ is already in $Y$. 
If $a>c$, then the fact that $a\ge h_i$ combined with 
the assumption in \eqref{eq:assumtion} gives
\[
 h_i\le a\le n+1-j.
\]
By \eqref{eq:comp}, we have $h_i\le h_j$. 
Since $(c,j)\in Y$, we know $c\ge h_j$, hence $c\ge h_i$, 
so $(c,i)\in Y$. This proves that $(\min(a,c),i)\in Y$.

For the second point, if $a\le c$, then $(\max(a,c),j)=(c,j)$ 
is already in $Y$. If $a>c$, then $c\ge h_j$ because $(c,j)\in Y$, 
and therefore $a>c\ge h_j$, so $(a,j)\in Y$ whenever $(a,j)\in \mathcal{Q}(n)$. 
Thus $(\max(a,c),j)\in Y$ whenever it belongs to $\mathcal{Q}(n)$.
This completes the proof.
\end{proof}

Let $Y\in\mathcal{R}'(n)$ and $(h_n,\dots,h_1)$ be its
threshold sequence. 
Since row $1$ has available $x$-coordinates $0,1,\dots,n-1$, we have
$h_1\in\{0,1,\dots,n-1,\infty\}$.
We distinguish three cases.

{\bf Case 1: $h_1=\infty$.} The row with the $y$-coordinate $1$ is empty. 
Then the remaining rows, including the bottom row truncated at the 
penultimate symbol, satisfy exactly the conditions for an element
in $\mathcal{R}'(n-1)$. Hence this case contributes $r'_{n-1}$ objects.

{\bf Case 2: $h_1=0$.}
The row with the $y$-coordinate $1$ is full. Again, deleting the bottom 
row gives an element of $\mathcal{R}'(n-1)$ (and vice versa). Hence this case also
contributes $r'_{n-1}$ objects.

{\bf Case 3: $1\le h_1\le n-1$.}
Applying Lemma~\ref{lem:comparison} with $i=1$, we obtain:
for every $j$ with $1< j\le n+1-h_1$, we have
$h_j\ge h_1$ or $h_j=\infty$.
In particular, when $j=n+1-h_1$, 
the row with the $y$-coordinate 
$j$ has total width $n-j=h_1-1$, so it cannot contain 
a threshold $\ge h_1$. This yields $h_{n+1-t}=\infty$.
Hence we can split the rows (ignoring the bottom row) 
into two independent blocks:
\begin{itemize}
\item the \emph{lower block}, consisting of all
rows with the $y$-coordinates $1,2,\dots,n-h_1$,
note that the thresholds in all these rows are at least $h_1$;
\item the \emph{upper block}, consisting of all
rows with the $y$-coordinates  $n+2-h_1,\dots,n$.
\end{itemize}
There is no interaction between the two blocks, because 
if $1\le i\le n-h_1$ and $n+2-h_1\le j\le n$, 
then every finite $h_i$ satisfies $h_i\ge h_1$, and hence
\[
 h_i+j\ge h_1+(n+2-h_1)=n+2,
\]
so Condition~$(\star)$ is vacuous across the two blocks.
Here is a pictorial illustration of this phenomenon,
for $n=9$ and $h_1=4$, with the non-trivial parts of the
two blocks colored in {\color{violet}violet} and {\color{teal}teal}.
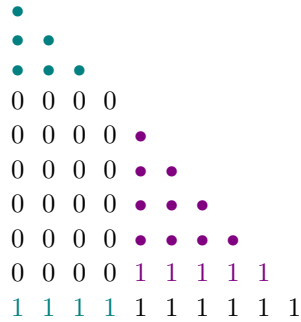
\begin{displaymath}
\xymatrix@R=0.2mm@C=0.2mm{
{\color{teal}\bullet}&&&&&&&&&\\
{\color{teal}\bullet}&{\color{teal}\bullet}&&&&&&&&\\
{\color{teal}\bullet}&{\color{teal}\bullet}&{\color{teal}\bullet}&&&&&&&&\\
0&0&0&0&&&&&&\\
0&0&0&0&{\color{violet}\bullet}&&&&&\\
0&0&0&0&
{\color{violet}\bullet}&{\color{violet}\bullet}&&&&\\
0&0&0&0&{\color{violet}\bullet}&{\color{violet}\bullet}&{\color{violet}\bullet}&&&\\
0&0&0&0&
{\color{violet}\bullet}&{\color{violet}\bullet}&{\color{violet}\bullet}&
{\color{violet}\bullet}&&\\
0&0&0&0&{\color{violet}1}&{\color{violet}1}&
{\color{violet}1}&{\color{violet}1}&{\color{violet}1}&\\
{\color{teal}1}&{\color{teal}1}&{\color{teal}1}&{\color{teal}1}&1&1&1&1&1&1\\
}
\end{displaymath}
It follows that the possible elements in the upper block
(combined with an appropriate part of the bottom row) are in bijection with
the elements in $\mathcal{R}(h_1-1)$ while the possible elements in
the non-trivial part of the lower block are in bijection 
with the elements in $\mathcal{R}(n-h_1-1)$.
Therefore this case contributes $R'_{n-h_1-1}\,R'_{h_1-1}$ elements.

Combining the three cases yields the recurrence
\begin{equation}\label{eq:bnrec}
R'_n=2R'_{n-1}+\sum_{t=1}^{n-1} R'_{t-1}R'_{n-t-1}\qquad (n\ge 1),
\end{equation}
which is one of the ways to write the usual Catalan recurrence.
The claim of the proposition follows.
\end{proof}

\subsection{Connection to A137842}\label{s9.4}

Let $\mathcal{R}''(n)$ denote the set of all 
elements $Y\in \mathcal{R}(n)$, where 
$Y=(Y_n,\dots,Y_0)$,
such that $Y_0=\big((2\mathbb{Z}+1)\times \{0\}\big)\cap \mathcal{Q}(n)$.
In other words, the bottom row of $Y$ consists of all 
elements of $\mathcal{Q}(n)$ whose first coordinate is odd
and second coordinate is zero. Set $R''_n:=|\mathcal{R}''(n)|$.
It is easy to check by a direct computation that 
the first values of $R''_n$ are $1,\,1,\,2,\,4,\,10,\,24,\,66$.

\begin{proposition}\label{prop:anotherconnection}
The numbers $(R''_n)_{n\ge 0}$ satisfy $R''_0=R''_1=1$,
and, for every $n\ge 2$,
\begin{equation}\label{eq:mainrec}
 R''_n
 =
 2R''_{n-1}
 +
 \sum_{i=1}^{\lfloor (n-2)/2\rfloor} R''_{2i}\,R''_{n-2i-1}.
\end{equation}
\end{proposition}

\begin{proof}
The proof is very similar to that of Proposition~\ref{prop:catalan}.
To start with, from Condition~$(\star)$ it follows immediately that
$Y\in \mathcal{R}''(n)$ contains no points for which the
first coordinate is even.

The next step is that, similarly to Proposition~\ref{prop:catalan},
each row $Y_i$, for $Y\in \mathcal{R}''(n)$, is either empty or 
an odd suffix (contains all odd points starting from some position). 
Therefore we can split $\mathcal{R}''(n)$ according
to the minimal element  in the row $Y_1$. 

If $Y_1$ is empty, we delete it and the last element in the bottom row
and see that such elements contribute $R''_{n-1}$. If $Y_1$ is 
odd-full, we delete the bottom row and see that such elements 
again  contribute $R''_{n-1}$.

It remains to consider the case when the smallest element in $Y_1$
is $(2i+1,1)$, for some $i\geq 1$. Exactly the same argument as in
Proposition~\ref{prop:catalan} splits $Y$ into an upper and a lower block
where the upper has $R''_{2i}$ possibilities while the lower has
$R''_{n-2i-1}$ possibilities. The application of the Rule of Product
followed by the Rule of Sum finishes the proof.
\end{proof}

Using the recurrence, we can compute more initial terms of
$R''_{n}$:
\[
1,\ 1,\ 2,\ 4,\ 10,\ 24,\ 66,\ 172,\ 498,\ 1360,\ 4066,\ 11444,\ 34970,\ 100520,\dots
\]
This matches the sequence A137842 on \cite{OEIS}
(see also its even-odd splits  A027307, A032349).
The sequence A137842 is defined as follows:
number of paths from $(0,0)$ if $n$ is even, or from $(2,1)$ 
if $n$ is odd, to $(3n,0)$ that stay in first quadrant 
(but may touch horizontal axis) and where each step is 
$(2,1)$, $(1,2)$ or $(1,-1)$. We can split such paths 
according to the first touch of the horizontal axis,
which gives the same recurrence as in 
Proposition~\ref{prop:anotherconnection}.
Therefore $R''_{n}$ is, indeed, given by the sequence A137842.

Of course, instead of looking at odd points in the bottom row,
we can look at even points, but this ultimately 
gives the same sequence, up to shift.

\subsection{Convex topologies on an integer interval}\label{s9.5}

The sequence A234268 in \cite{OEIS}
is defined as the number of convex topologies 
on an $n$-point totally ordered set. A topology is called {\em convex}
if it has a base of convex sets, or equivalently, if each point 
has a neighborhood base of convex sets, see \cite{CR}.
For a finite totally ordered set, convex subsets are exactly the intervals.
As mentioned in Subsection~\ref{s2.9}, there is a natural bijection
between the elements of $\mathcal{Q}(n)$ and intervals 
of a finite totally ordered $n+1$-element set. Via this bijection,
this totally ordered $n+1$-element set itself corresponds to the origin
in $\mathcal{Q}(n)$.

Consequently, mapping a topology to the set of intervals contained in 
it defines a bijection between the set of convex topologies 
on an $n+1$-point totally ordered set and $\mathcal{P}(n)$.
Here we emphasize two points:
\begin{itemize}
\item the origin must be included due to axioms of topology,
therefore we have a bijection with $\mathcal{P}(n)$ and not 
with $\mathcal{R}(n)$.
\item the principal point of the connection is that Condition~$(\star)$,
on the topological side, represents the fact that the union of
two intervals is an interval precisely when the original two intervals
intersect or are adjacent, that is, precisely when Condition~$(\star)$
is triggered.
\end{itemize}
Theorem~A thus provides a recurrence for $|\mathcal{P}(n)|$.
Since we established that $|\mathcal{P}(n)|$ and $|\mathcal{R}(n)|$
have the same growth, see Subsection~\ref{s7.2}, 
our result from Section~\ref{s7}
provide information for the growth of A234268.

The initial values of A234268 available at \cite{OEIS} are:
\begin{displaymath}
1,\, 4,\, 21,\, 129,\, 876,\, 6376,\, 48829,\, 388771,\, 3191849,\, 26864936.
\end{displaymath}
Theorem~A allows for quick computation of further values of this sequence,
see Section~\ref{s5} for details.

\section{Lattice structure on $\mathcal{R}(n)$}\label{s12}

\subsection{$\mathcal{R}(n)$ as a poset with respect to inclusions}\label{s12.1}

The set $\mathbf{2}^{\mathcal{Q}(n)}$ has the natural structure of a poset
with respect to inclusions. This poset is a lattice, where the meet is given
by intersection and the join is given by union. The set 
$\mathcal{R}(n)$ inherits the structure of a poset 
from $\mathbf{2}^{\mathcal{Q}(n)}$. Directly from the definitions it is 
clear that $\mathcal{R}(n)$ is closed with respect to intersections,
in particular, it is a meet-subsemilattice of $\mathbf{2}^{\mathcal{Q}(n)}$.

Given $X\in \mathbf{2}^{\mathcal{Q}(n)}$, set
\begin{displaymath}
\overline{X}:=\bigcap_{X\subset Y\in \mathcal{R}(n)} Y.
\end{displaymath}
Then $\overline{X}\in \mathcal{R}(n)$ by the previous paragraph.
Moreover, $\overline{\cdot}:\mathbf{2}^{\mathcal{Q}(n)}\to \mathcal{R}(n)$
is a closure operator. The set $\mathcal{R}(n)$ is not closed under 
the union of sets, in general
(for example, $\{(1,0)\},\{(0,1)\}\in \mathcal{R}(2)$ while
$\{(1,0),(0,1)\}\not\in \mathcal{R}(2)$). However, 
for $X,Y\in \mathcal{R}(n)$, the element $\overline{X\cup Y}$
is the join of $X$ and $Y$. In particular, $\mathcal{R}(n)$ is a lattice.
This lattice is not distributive in general. For example, for
$\mathcal{R}(1)$, we have 
\begin{displaymath}
\{(0,0)\}\cap\overline{\{(1,0)\}\cup\{(0,1)\}}=
\{(0,0)\}\neq \varnothing =
\overline{\{(0,0)\}\cap\{(1,0)\}}\cap\overline{\{(0,0)\}\cap\{(0,1)\}}.
\end{displaymath}

The poset $\mathcal{R}(n)$ has $\varnothing$ as the minimum element
and $\mathcal{Q}(n)$ as the maximum element.

\subsection{Atoms}\label{s12.2}

Each singleton in $\mathbf{2}^{\mathcal{Q}(n)}$ clearly  belongs to 
$\mathcal{R}(n)$. Since these are exactly 
the atoms of $\mathbf{2}^{\mathcal{Q}(n)}$ it follows that they 
are also exactly the atoms of $\mathcal{R}(n)$.

As a lattice, $\mathcal{R}(n)$ is, clearly, both atomic and atomistic.
The join-irreducible elements are the atoms.

\subsection{Coatoms}\label{s12.3}

For $k=1,2,\dots,n$, define
\begin{displaymath}
\mathtt{C}(n,k)=\{(k,0),(k,1),\dots,(k,n-k)\},\quad 
\mathtt{R}(n,k)=\{(0,k),(1,k),\dots,(n-k,k)\}. 
\end{displaymath}
Note that, for $n=0$, we have $\mathcal{R}(0)=\mathcal{Q}(0)$
with $\varnothing$ being the unique coatom and 
$\mathcal{Q}(0)$ being the unique atom.

\begin{proposition}\label{prop:coatoms}
For $n>0$, the poset  $\mathcal{R}(n)$ has exactly $2n$ coatoms.
These are the elements $\mathcal{Q}(n)\setminus \mathtt{C}(n,i)$
and $\mathcal{Q}(n)\setminus \mathtt{R}(n,k)$, for $k=1,2,\dots,n$.
\end{proposition}

The proof below is a significant rewrite of an
original proposal  by Microsoft Copilot which contained quite a few
gaps.

\begin{proof}
We assume that $n>0$.

Let us start by proving that each element of the form 
$\mathcal{Q}(n)\setminus \mathtt{C}(n,i)$
and of the form $\mathcal{Q}(n)\setminus \mathtt{R}(n,k)$, for $k=1,2,\dots,n$,
is a coatom in $\mathcal{R}(n)$. We prove this for the elements
$Y^{(k)}:=\mathcal{Q}(n)\setminus \mathtt{C}(n,k)$, for
$\mathcal{Q}(n)\setminus \mathtt{R}(n,k)$ the proof is similar.

To start with, we claim that 
$Y^{(k)}\in \mathcal{R}(n)$. Indeed, take $(x,y),(u,v)\in Y^{(k)}$.  
Then neither point has first coordinate equal to $k$.  
Therefore both $\min(x,u)\neq k$ and $\max(x,u)\neq k$ and thus
both 
\[
(\min(x,u),\min(y,v))\quad\text{and}\quad (\max(x,u),\max(y,v))
\]
(if the latter belongs to $\mathcal{Q}(n)$) also lie outside 
$\mathtt{C}(n,k)$, hence belong to $Y^{(k)}$.  This shows that $Y^{(k)}\in \mathcal{R}(n)$.

If $k=n$, then the corresponding $\mathcal{Q}(n)\setminus \mathtt{C}(n,k)$
is a coatom in $\mathbf{2}^{\mathcal{Q}(n)}$, hence also in
$\mathcal{R}(n)$. Now assume $k<n$ and let $Z\in \mathcal{R}(n)$ 
be such that $Y^{(k)}\subsetneq Z\subseteq \mathcal{Q}(n)$.  
Then $Z$ contains $(k,j)\in \mathtt{C}(n,k)$, for some $j$.
We claim that this forces $\mathtt{C}(n,k)\subseteq Z$, hence $Z=\mathcal{Q}(n)$,
which is a contradiction.
Indeed, starting from $(k,j)$ one can propagate both upward and downward.
\begin{itemize}
\item If $j<n-k$, then $(k-1,j+1)\in Y_k\subseteq Z$, and
\[
\max\bigl((k,j),(k-1,j+1)\bigr)=(k,j+1)\in \mathcal{Q}(n).
\]
Moreover
\[
\max(k,k-1)+\max(j,j+1)=k+(j+1)\le n+1.
\]
Hence Condition~$(\star)$ forces $(k,j+1)\in Z$.
\item If $j>0$, then $(k+1,j-1)\in Y_k\subseteq Z$ 
(note that $(k+1,j-1)\in \mathcal{Q}(n)$ because
$(k,j)\in \mathcal{Q}(n)$ implies $k+j\le n$), and
\[
\min\bigl((k,j),(k+1,j-1)\bigr)=(k,j-1).
\]
Also
\[
\max(k,k+1)+\max(j,j-1)=(k+1)+j\le n+1.
\]
Therefore $(k,j-1)\in Z$.
\end{itemize}
By iterating these two steps, every point of $\mathtt{C}(n,k)$ belongs to $Z$.  
This shows that $Y^{(k)}$ is, indeed,  coatom in $\mathcal{R}(n)$.

Now let us show that every coatom  in $\mathcal{R}(n)$ is of one of the above forms.
Let $Y\in \mathcal{R}(n)$ be a coatom and set $M:=\mathcal{Q}(n)\setminus Y$.
Then $M\neq\varnothing$. Choose $(a,b)\in M$ with $a+b$ minimal.
We claim that either  $a=0$ or $b=0$. Indeed, if $a>0$ and $b>0$,
then both $(a-1,b)$ and $(a,b-1)$ are in $Y$, so 
Condition~$(\star)$ forces $(a,b)\in Y$, a contradiction.

By symmetry, we may therefore assume that our minimal missing point 
is of the form $(k,0)$ with $0\le k\le n$. Assume that the missing point is
$(0,0)$. We know that, adding $(0,0)$, we get a new element in 
$\mathcal{R}(n)$. Since $Y$ is a coatom, this forces the equality
$Y=\mathcal{Q}(n)\setminus\{(0,0)\}$. From 
Condition~$(\star)$, we have $\mathcal{Q}(n)\setminus\{(0,0)\}\in 
\mathcal{R}(n)$ only if $n=0$ and this case is excluded from our consideration.
Therefore our minimal missing point 
is of the form $(k,0)$ with $1\le k\le n$. 

If $k=n$, then $Y\subset Y^{(n)}$ and since both $Y$ and $Y^{(n)}$ are coatoms,
we obtain $Y=Y^{(n)}$. Therefore it remains to consider the case $k<n$.
In this case let us assume that $Y\not\subset Y^{(k)}$, that is, there is
some minimal $m$ such that $(k,m)\in Y$. In other words,
we know that $(k,m)\in Y$, for some $0<m\leq n-k$, and that 
$(k,i)\not\in Y$, for all $0\leq i<m$. From Condition~$(\star)$
it follows that $(n+1-m,i)\not\in Y$, for all $0\leq i<m$.
Here is the illustration:
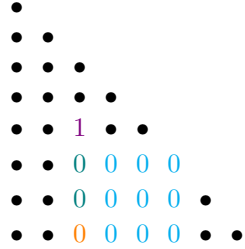
\begin{displaymath}
\xymatrix@R=0.3mm@C=0.3mm{
\bullet&&&&&&&\\
\bullet&\bullet&&&&&&\\
\bullet&\bullet&\bullet&&&&&\\
\bullet&\bullet&\bullet&\bullet&&&&\\
\bullet&\bullet&{\color{violet}1}&\bullet&\bullet&&&\\
\bullet&\bullet&{\color{teal}0}&{\color{cyan}0}
&{\color{cyan}0}&{\color{cyan}0}&&\\
\bullet&\bullet&{\color{teal}0}&{\color{cyan}0}
&{\color{cyan}0}&{\color{cyan}0}&\bullet&\\
\bullet&\bullet&{\color{orange}0}&{\color{cyan}0}&
{\color{cyan}0}&{\color{cyan}0}&\bullet&\bullet\\
}
\end{displaymath}
Here the {\color{orange}orange} $0$ is the point $(k,0)$
and the {\color{violet}violet} $1$ is the point $(k,m)$
so that we have the {\color{teal}teal} $0$'s in between.
The {\color{cyan}cyan} part is then forced by Condition~$(\star)$
and the rightmost cyan column is exactly $\mathtt{C}(n,n+1-m)$.
Therefore $Y\subset Y^{(n+1-m)}$ and, consequently,
$Y=Y^{(n+1-m)}$ due to our assumption that $Y$ is a coatom.
This completes the proof.
\end{proof}

\subsection{Meet irreducible elements}\label{s12.5}

Since $\mathcal{R}(n)$ is finite, an element $Y\in \mathcal{R}(n)$ is
{meet-irreducible} if and only if $Y\neq \mathcal{Q}(n)$ and $Y$ 
has a unique cover. For arbitrary integers $1\le a\le b\le n$, define
\[
V_{a,b}:=\{(x,y)\in \mathcal{Q}(n): a\le x\le b,\ 0\le y\le n-b\}.
\]
For integers $0\le c\le d\le n$, define
\[
H_{c,d}:=\{(x,y)\in \mathcal{Q}(n): 0\le x\le n-d,\ c\le y\le d\}.
\]
Thus $V_{a,b}$ is a rectangle resting on the $x$-axis whose 
upper-right corner lies on the
boundary line $x+y=n$, while $H_{c,d}$ similarly 
rests on the $y$-axis.

\begin{theorem}\label{thm:meetirr}
The meet-irreducible elements of $\mathcal{R}(n)$ are exactly the sets
\[
\mathcal{Q}(n)\setminus V_{a,b}\qquad (1\le a\le b\le n)
\quad
\text{ and }
\quad
\mathcal{Q}(n)\setminus H_{c,d}\qquad (0\le c\le d\le n).
\]
In particular, the number of meet-irreducible elements is
$(n+1)^2$.
\end{theorem}

The statement of this theorem is due to Microsoft Copilot,
however, the ``proof'' which Copilot produced contained several
serious gaps and required a significant amount of additional work. 
At the same time, the general strategy of the proof follows the one proposed
by Copilot.

\begin{proof}
We split the proof into a number of steps.

\textbf{Step 1: the displayed sets belong to $\mathcal{R}(n)$.}
Fix some $1\le a\le b\le n$ and consider the set 
$Y:=\mathcal{Q}(n)\setminus V_{a,b}$. Let $(x,y),(u,v)\in Y$ and suppose
that we have the inequality $\max(x,u)+\max(y,v)\le n+1$.
We need to show that the two coordinate-wise 
extrema also lie in $Y$ whenever Condition~$(\star)$ applies.

The only way the coordinate-wise minimum or 
maximum could fail to lie in $Y$ would be for that point
to lie in the deleted rectangle $V_{a,b}$. 
But any point of $V_{a,b}$ has first coordinate in the
interval $[a,b]$ and second coordinate at most $n-b$. 
Since every point of $Y$ lies either strictly
left of the strip $a\le x\le b$, or strictly right of 
it, or strictly above the height $n-b$, a
coordinate-wise minimum or maximum of two points of $Y$ 
cannot enter the rectangle $V_{a,b}$ unless
one of the two original points was already in $V_{a,b}$.
Hence $Y\in \mathcal{R}(n)$. The verification for 
$\mathcal{Q}(n)\setminus H_{c,d}$ is similar.

\textbf{Step 2: each displayed set is meet-irreducible.}
We prove the claim for the vertical family 
$\mathcal{Q}(n)\setminus V_{a,b}$.
For the horizontal family 
$\mathcal{Q}(n)\setminus H_{c,d}$, 
the proof is similar. 

Let
\[
Y:=\mathcal{Q}(n)\setminus V_{a,b}\qquad (1\le a\le b\le n).
\]
Consider the set
\[
S_a:=\{(a,j):0\le j\le n-b\},
\]
which is the left-hand column of the deleted rectangle. We claim that
\[
Y^+:=Y\cup S_a
\]
belongs to $\mathcal{R}(n)$ and is the unique cover of $Y$.

That $Y^+\in \mathcal{R}(n)$ follows from Step~1 and the observation
that $Y^+=\mathcal{Q}(n)\setminus V_{a+1,b}$, if $a<b$, while, for $a=b$, 
we have $Y^+=\mathcal{Q}(n)$.

Now let $Z\in \mathcal{R}(n)$ satisfy $Y\subsetneq Z$. 
Then $Z$ contains at least one point of the deleted
rectangle $V_{a,b}$. We show that necessarily $S_a\subseteq Z$.
Let $(u,v)\in Z\cap V_{a,b}$.

If $u>a$, consider the point $(u-1,n-b+1)$. 
This point belongs to $Y$ (it lies just above the
deleted rectangle), and
\[
\max(u,u-1)+\max(v,n-b+1)=u+n-b+1\le b+n-b+1=n+1.
\]
Therefore Condition~$(\star)$ forces
\[
(u-1,v)=\min\bigl((u,v),(u-1,n-b+1)\bigr)\in Z.
\]
Repeating this step moves us left until we reach $(a,v)\in Z$.

\item From $(a,v)\in Z$, if $v<n-b$ then the point $(a-1,v+1)$ belongs to $Y$ (here we use $a\ge1$), and
\[
\max(a,a-1)+\max(v,v+1)=a+v+1\le a+(n-b)+1\le n+1.
\]
Hence Condition~$(\star)$ forces
\[
(a,v+1)=\max\bigl((a,v),(a-1,v+1)\bigr)\in Z.
\]
Thus we can propagate upward and obtain $(a,j)\in Z$ for every $j$ with $v\le j\le n-b$.

Similarly, if $v>0$, then $(b+1,v-1)\in Y$ when $b<n$, and
\[
\max(a,b+1)+\max(v,v-1)=b+1+v\le b+1+(n-b)=n+1,
\]
so Condition~$(\star)$ yields $(a,v-1)\in Z$. Repeating, 
we obtain $(a,j)\in Z$ for all $0\le j\le v$.
When $b=n$, necessarily $v=0$, so there is nothing further to prove in this direction.

The above implies $S_a\subseteq Z$. Therefore every proper closed superset 
of $Y$ contains $Y^+=Y\cup S_a$. Since $Y^+\in \mathcal{R}(n)$ 
it thus follows that the cover $Y^+$ of $Y$ is unique and therefore
$Y$ is meet-irreducible.

\textbf{Step 3: every meet-irreducible has one of the displayed forms.}

Let $Y\in \mathcal{R}(n)$ be meet-irreducible and put
$M:=\mathcal{Q}(n)\setminus Y$.
Since $Y\ne \mathcal{Q}(n)$, the set $M$ is nonempty.
Because $Y$ is meet-irreducible, it has a unique cover, 
equivalently, for every $p\in M$, the set
$\overline{Y\cup\{p\}}\in \mathcal{R}(n)$
contains the same minimal proper closed extension of $Y$. 
We claim that this forces $M$ to be a single axis-anchored rectangle.

To prove our claim,
consider $M$ as a poset with respect to the usual partial order
$(x,y)\leq (u,v)$ if and only if $x\leq u$ and $y\leq v$.

Assume that $M$ contains a unique maximal element $(u,v)$.
Then we have that both $\{(x,y)\in\mathcal{Q}(n)\,:\, x>u\}$
and $\{(x,y)\in\mathcal{Q}(n)\,:\, y>v\}$ belong to $Y$.
If $u+v<n$, then $(u+1,v)\in Y$ and $(u,v+1)\in Y$ implying 
$(u,v)\in Y$ by Condition~$(\star)$, a contradiction.
Therefore $u+v=n$.

If $u=n$, then $v=0$ and hence the whole $M$ is contained in the 
bottom row of $\mathcal{Q}$. From Condition~$(\star)$ it follows
directly that $M$ must be a suffix of the bottom row and therefore
have the form $V_{j,n}$, for some $j$. Similarly,
if $v=n$ and $u=0$, we get $M=H_{j,n}$, for some $j$.

Now assume $u,v\neq n$ and thus also $u,v\neq 0$.
Let us now look at the hook
\begin{displaymath}
\{(0,v),(1,v),\dots,(u,v)\}\cup 
\{(u,v),(u,v-1),\dots,(u,0)\}.
\end{displaymath}
Because of our assumptions, the whole triangular
part of $\mathcal{Q}(n)$ above this hook is non-empty
and belongs to $Y$. Similarly, the whole triangular
part of $\mathcal{Q}(n)$ to the right of this hook is non-empty
and belongs to $Y$. 

Assume first that some $(j,v)$ from this hook is contained
in $Y$ and choose $j$ to be maximal with this property.
Applying Condition~$(\star)$ to all $(j,s)$,
for $0\leq s<j$ and the
left column of the triangular part of $\mathcal{Q}(n)$ 
to the right of our hook, we obtain that 
$Y$ contains the whole column which has the first coordinate $j$.
Since $Y$ also contains the row with second coordinate $v+1$,
it follows that $M$ is contained in $V_{j+1,u}$. 
At the same time, if $Y$ contains some point in $V_{j+1,u}$,
applying Condition~$(\star)$ to it and $(j,v)$ we get 
that $Y$ contains $(j',v)$, for $j'>j$, a contradiction.
Therefore $Y=\mathcal{Q}(n)\setminus V_{j+1,u}$ in this case.
Similarly, if $Y$ contains some point from the column in our
hook, we obtain that $Y=\mathcal{Q}(n)\setminus H_{j,v}$,
for some $j$.

Now consider the case when the whole hook is contained in $M$.
We want to prove that in this case $M=H_{0,v}$.
Assume, towards contradiction, that $(x,y)\in Y$ is 
such that $x<u$ and $y<v$. Applying Condition~$(\star)$
to $(x,y)$ and the upper and the right bounds of our hook, 
we deduce that  $(x',y')\in Y$, for all $0\leq x'\leq x$
and $0\leq y'\leq y$. It is easy to see that Condition~$(\star)$
implies existence of a unique $(x,y)\in Y$  
such that $x<u$ and $y<v$ and, moreover, such that
$(x',y')\in Y$ implies $x'\leq x$ and $y'\leq y$, for
all $x'<u$ and $y'<v$. This means that our $Y$ looks as follows:
\begin{displaymath}
\xymatrix@R=0.3mm@C=0.3mm{
1&&&&&&&\\
1&1&&&&&&\\
1&1&1&&&&&\\
1&1&1&1&&&&\\
1&1&1&1&1&&&\\
0&0&0&0&0&{\color{violet}0}&&\\
1&1&{\color{teal}1}&0&0&0&1&\\
1&1&1&0&0&0&1&1\\
}
\end{displaymath}
Here $(x,y)$ is the {\color{teal}teal} $1$
while $(u,v)$ is the {\color{violet}violet} $0$.
The important thing here is that we can add to our
South-East rectangle of $1$'s one more column {\color{magenta}on the right}
or one more column {\color{orange}above}:
\begin{displaymath}
\xymatrix@R=0.3mm@C=0.3mm{
1&&&&&&&\\
1&1&&&&&&\\
1&1&1&&&&&\\
1&1&1&1&&&&\\
1&1&1&1&1&&&\\
0&0&0&0&0&{\color{violet}0}&&\\
1&1&{\color{teal}1}&{\color{magenta}1}&0&0&1&\\
1&1&1&{\color{magenta}1}&0&0&1&1\\
}\qquad
\xymatrix@R=0.3mm@C=0.3mm{
1&&&&&&&\\
1&1&&&&&&\\
1&1&1&&&&&\\
1&1&1&1&&&&\\
1&1&1&1&1&&&\\
{\color{orange}1}&{\color{orange}1}&{\color{orange}1}&0&0&{\color{violet}0}&&\\
1&1&{\color{teal}1}&0&0&0&1&\\
1&1&1&0&0&0&1&1\\
}
\end{displaymath}
This produces two elements in $\mathcal{R}(n)$
whose meet is $Y$, which means that $Y$ is not meet irreducible.
Hence we deduce that $Y=\mathcal{Q}(n)\setminus H_{0,v}$.

Assume now that $M$ contains at least two different
maximal elements $p:=(f,g)$ and $q:=(s,t)$ with respect to our 
partial order. Similarly to the above, we have $f+g=s+t=n$.
In particular, $f\neq s$ and $g\neq t$.
We can choose $p$ and $q$ such that $f$ is minimal possible
while $s$ is maximal possible. Define:
\begin{itemize}
\item $R_p$ to be the row of $\mathcal{Q}(n)$ containing $p$;
\item $R_q$ to be the row of $\mathcal{Q}(n)$ containing $q$;
\item $C_p$ to be the column of $\mathcal{Q}(n)$ containing $p$;
\item $C_q$ to be the column of $\mathcal{Q}(n)$ containing $q$.
\end{itemize}
Due to Condition~$(\star)$, we either have
$R_p\subset M$ or $C_p\subset M$ (or both)
and also we either have
$R_q\subset M$ or $C_q\subset M$ (or both).

Assume first that $R_p\subset M$ and $C_q\subset M$.
Define $Y_q:=\mathcal{Q}(n)\setminus R_p$
and $Y_p:=\mathcal{Q}(n)\setminus C_q$.
Both are, clearly, in $\mathcal{R}(n)$.
By construction, we also have 
$Y\subset Y_p$, $Y\subset Y_q$,
$p\in Y_p$, $q\in Y_q$, $p\not\in Y_q$
and $q\not\in Y_p$.
This implies that $Y$ has at least two different covers and hence
is not meet-irreducible. 

Assume next that $R_p\subset M$ but $C_q\not\subset M$.
Then $R_q\subset M$.
Define $Y_p:=\mathcal{Q}(n)\setminus R_q$
and $Y_q:=\mathcal{Q}(n)\setminus R_p$.
Both are in $\mathcal{R}(n)$.
By construction, we have 
$Y\subset Y_p$, $Y\subset Y_q$,
$p\in Y_p$, $q\in Y_q$, $p\not\in Y_q$
and $q\not\in Y_p$.
Hence $Y$ has at least two different covers and thus
is not meet-irreducible.  Similarly one deals with the case 
$R_p\not\subset M$ and $C_q\subset M$.

Finally, consider the case of both $R_p\not\subset M$ but $C_q\not\subset M$.
Then $C_p\subset M$ and $R_q\subset M$.
Define $Y_p:=\mathcal{Q}(n)\setminus R_q$
and $Y_q:=\mathcal{Q}(n)\setminus C_p$.
Both are, clearly, in $\mathcal{R}(n)$.
By construction, we also have 
$Y\subset Y_p$, $Y\subset Y_q$,
$p\in Y_p$, $q\in Y_q$, $p\not\in Y_q$
and $q\not\in Y_p$.
This implies that $Y$ has at least two different covers and hence
is not meet-irreducible.  

This completes the proof of the fact
that the elements in the formulation of our theorem are exactly 
the meet-irreducible elements.

\textbf{Step 4: enumeration.}
Note that 
the vertical and horizontal families are disjoint with the chosen 
parameter ranges. Indeed, the vertical family is indexed by 
$1\le a\le b\le n$, while the horizontal family is indexed by $0\le c\le d\le n$.
Hence the count of meet-irreducibles is exactly
\[
\frac{n(n+1)}2+\frac{(n+1)(n+2)}2=(n+1)^2.
\]
This completes the proof.
\end{proof}

\vspace{2mm}

\noindent
Department of Mathematics, Uppsala University, Box. 480,
SE-75106, Uppsala,\\ SWEDEN, email: {\tt mazor\symbol{64}math.uu.se}

\newpage

\section{Appendix: The SageMath code}

Here is the SageMath code written by Copilot for computation 
of $R_n$:

\begin{tcolorbox}[colback=gray!10,colframe=gray!50]
\begin{verbatim}
def compute_abRP(Max):
    A = {}
    B = {}

    for n in range(-2, Max + 1):
        A[n] = {}
        B[n] = {}

    A[-2][0] = ZZ(1)
    A[-1][0] = ZZ(1)
    A[0][0]  = ZZ(1)
    A[0][1]  = ZZ(1)

    B[0][1] = ZZ(1)

    def a_get(n, k):
        if k < 0 or n < -2:
            return ZZ(0)

        if n >= 0 and k > n + 1:
            return ZZ(0)

        if n < 0 and k != 0:
            return ZZ(0)

        return A.get(n, {}).get(k, ZZ(0))

    def b_get(n, k):
        if k <= 0:
            return ZZ(0)
        if n < 0:
            return ZZ(0)
        if k > n + 1:
            return ZZ(0)

        return B.get(n, {}).get(k, ZZ(0))

    def R(n):
        if n < 0:
            return a_get(n, 0)
        return sum(a_get(n, k) for k in range(0, n + 2))

    def P(n):
        if n < 0:
            return ZZ(0)
        return sum(b_get(n, k) for k in range(1, n + 2))
\end{verbatim}
\end{tcolorbox}
\newpage

\begin{tcolorbox}[colback=gray!10,colframe=gray!50]
\begin{verbatim}

    for n in range(1, Max + 1):

        for k in range(1, n + 2):
            val = a_get(n - 1, k - 1)

            extra = ZZ(0)
            for m in range(1, n + 1):
                upper_r = min(m - 1, k - 1)
                for r in range(0, upper_r + 1):
                    aval = a_get(m - 2, r)
                    if aval == 0:
                        continue

                    p_start = k - 1 - r
                    p_end   = n - m + 1

                    inner = ZZ(0)
                    for p in range(p_start, p_end + 1):
                        inner += b_get(n - m, p)

                    extra += aval * inner

            val += extra
            if val != 0:
                B[n][k] = val

        A[n][0] = R(n - 1)

        for k in range(1, n + 2):
            val = b_get(n, k)

            extra = ZZ(0)
            for j in range(0, n):
                bjk = b_get(j, k)
                if bjk == 0:
                    continue

                inner = ZZ(0)
                for i in range(0, n - j):
                    inner += a_get(n - j - 2, i)

                extra += bjk * inner

            val += extra
            if val != 0:
                A[n][k] = val

    return A, B, R, P
\end{verbatim}
\end{tcolorbox}
\newpage

\begin{tcolorbox}[colback=gray!10,colframe=gray!50]
\begin{verbatim}

def print_abRP(Max):
    A, B, R, P = compute_abRP(Max)

    print("Values for n = 0,...,%s" % Max)
    print("")

    for n in range(0, Max + 1):
        a_row = [A[n].get(k, ZZ(0)) for k in range(0, n + 2)]
        b_row = [B[n].get(k, ZZ(0)) for k in range(0, n + 2)]

        print("n = %s" % n)
        print("  a(%s, .) = %s" % (n, a_row))
        print("  b(%s, .) = %s" % (n, b_row))
        print("  R(%s) = %s" % (n, R(n)))
        print("  P(%s) = %s" % (n, P(n)))
        print("")

Max = 10
print_abRP(Max)
\end{verbatim}
\end{tcolorbox}
 

\begin{thebibliography}{999999}
%

\bibitem[BPS96]{BPS}
Brightwell, G.; Pr{\"o}mel, H.-J.; Steger, A.
The average number of linear extensions of a partial order.
J. Combin. Theory Ser. A {\bf 73} (1996), no. 2, 193--206.

\bibitem[CR15]{CR}
Clark, T.; Richmond, T.
The number of convex topologies on a finite totally ordered set.
Involve {\bf 8} (2015), no. 1, 25--32.

\bibitem[DS]{Sage}
{Developers, The~Sage}.
{S}agemath, the {S}age {M}athematics {S}oftware {S}ystem.
({V}ersion 10.9),
{2026}, {{\tt https://www.sagemath.org}}.

\bibitem[FS09]{FS}
Flajolet, P. and Sedgewick, R.
Analytic Combinatorics.
Cambridge University Press, 2009.

\bibitem[Ga72]{Ga}
Gabriel, P.
Unzerlegbare Darstellungen. I
Manuscripta Math. {\bf 6} (1972), 71–103; 
correction, ibid. {\bf 6} (1972), 309.

\bibitem[KS26]{KS}
Krause, H.; Stoye, B.
Multisets of finite intervals and a universal category of 
poset representations.
Preprint arXiv:2601.22649.

\bibitem[OEIS]{OEIS} 
The On-Line Encyclopedia of Integer Sequences.
OEIS Foundation Inc. (2026), www.oeis.org.

\end{thebibliography}
\end{document}